\documentclass[11pt]{amsart}
\usepackage[hmargin=2.5cm,vmargin=2.5cm]{geometry}
\usepackage{amsfonts}
\usepackage{amsthm}
\usepackage[english]{algorithm2e}
\usepackage{amsmath}
\usepackage{amssymb}
\usepackage{hyperref}
\usepackage[T1]{fontenc}
\newtheorem{Theorem}{Theorem}[section]
\newtheorem{theorem}[Theorem]{Theorem}
\newtheorem{definition}[Theorem]{Definition}
\newtheorem{remark}[Theorem]{Remark}
\newtheorem{example}[Theorem]{Example}

\newtheorem{proposition}[Theorem]{Proposition}
\newtheorem{corollary}[Theorem]{Corollary}

\title{Representations and $\mathcal{O}$-operators of Hom-(pre)-Jacobi-Jordan algebras}
\author[Sylvain Attan ]
       { Sylvain Attan}
\begin{document}
\maketitle
\begin{abstract}
Representations and $\mathcal{O}$-operators  of Hom-(pre)-Jacobi-Jordan algebras are introduced and studied. The anticommutator of a Hom-pre-Jacobi-Jordan algebra is a Hom-Jacobi-Jordan algebra and the left multiplication operator gives a representation of a Hom-Jacobi-Jordan
algebra.  The notion of matched pairs and Nijenhuis operators of Hom-(pre)-Jacobi-Jordan algebras are given and various relevant constructions  are obtained.
\end{abstract}

\section{Introduction}
Jacobi-Jordan algebras are commutative algebras satisfying the Jacobi identity. 
Their first interest is motivated by the fact that they constitute an interesting sub-category  of the well-referenced category of Jordan algebras which were introduced to explain some aspects of physics \cite{pjjvn}. These algebras seem appeared first in \cite{kaz}, where an example of infinite-dimensional solvable but not nilpotent Jacobi-Jordan algebra was given. Since then, different names are used for these algebraic structures, indeed they are called for the first time mock-Lie algebras in \cite{egmk}, where the corresponding operad appears in the list of quadratic cyclic operads with one generator whereas the terms Jordan algebras of nil index 3 \cite{sw}, Lie-Jordan algebras \cite{sonk} and  finally Jacobi-Jordan algebras in the recent paper \cite{dbaf} are used.  Thanks to the approach of Eilenberg nicely described in \cite{nj},  representations of Jacobi-Jordan algebras are introduced in \cite{pzu} where many facts and conjectures about these algebras are made. As Pre-Lie algebras (also called left-symmetric algebras, quasi-associative algebras, Vinberg algebras and so on), left(resp. right)pre-Jacobi-Jordan algebras  \cite{cedh} are apparently first introduced in \cite{absb} under the name left skew-symmetric (resp. right skew-symmetric) algebras. The study of some relevant properties such as bimodules, matched pairs is considered for these algebras \cite{cedh}. Moreover, a $2$-dimensional classification and some double constructions of these algebras are given.
Observe that there is a close relationship between pre-Jacobi-Jordan algebras and Jacobi-Jordan: a pre-Jacobi-Jordan algebra $(A, \cdot)$ gives rise to a Jacobi-Jordan algebra $(A, \ast )$ via the anticommutator
multiplication, which is called the subadjacent Jacobi-Jordan algebra. Furthermore, for a given pre-Jacobi-Jordan algebra $(A, \cdot)$,the
map $L : A\rightarrow gl(A), $ defined by 
$L_x y=x\cdot y$ for all $x, y \in A,$ gives rise to a representation of its
subadjacent pre-Jacobi-Jordan algebra.

The theory of Hom-algebras was  first initiated  by D. Larsson, S. D. Silvestrov and J. T. Hartwig \cite{HAR1},\cite{dlsds1}, \cite{dlsds2} with the introduction of Hom-Lie algebras . 
 It is known that any associative algebra is Lie-admissible i.e., the commutator algebra of any associative algebra is a Lie algebra. To be in adequacy with this fact in the theory of Hom-algebras, Hom-associative algebras were introduced \cite{MAK3} and it was shown that the commutator Hom-algebra of any Hom-associative algebra is a Hom-Lie algebra.    Since then, other types of Hom-algebras  such  as
 Hom-Novikov algebras, Hom-alternative
algebras, Hom-Jordan algebras or Hom-Malcev algebras are defined and discussed in \cite{am}, \cite{dy1}, \cite{dy2}.

 The main objective of this paper is  the study of representations of Hom-(pre)-Jacobi-Jordan alegbras (see  \cite{sa1}, \cite{sa2}, \cite{ysc}, \cite{ys} for the study of representations of other Hom-algebras) as well as the  one of $\mathcal{O}$-operators, also known as relative or generalized Rota-Baxter operators on these Hom-algebras. First introduced in \cite{gb}  by Baxter for associative algebras,  Rota-Baxter operators have several applications in
probability \cite{gb}, combinatorics \cite{pc},\cite{lgwk}, \cite{gcr} and quantum field theory \cite{acdk}.  Rota-Baxter operators of weight $0$ for Lie algebras \cite{lg}  were introduced in terms of the classical Yang-Baxter equation  and later on, Kupershmidt \cite{bak} defined $\mathcal{O}$-operators as generalized Rota-Baxter operators to understand classical Yang-Baxter equations and related integrable systems.

The paper is organized as follows. 
Section 2 is devoted to reminders of fundamental concepts. In Section 3, we introduce the notion of Hom-Jacobi-Jordan algebra, provide some properties and define the notion of a representation, $\mathcal{O}$-operators and matched pairs  of a Hom-Jacobi-Jordan algebra. We prove that any Hom-Jacobi-Jordan algebra is a Hom-Jordan algebra. Moreover, we develop some constructions theorems about representations, $\mathcal{O}$-operators and matched pairs for these Hom-algebras. Finally, the last section contains many relevant results. First, we introduce Hom-pre-Jacobi-Jordan algebras and prove that the anticommutator of a Hom-pre-Jacobi-Jordan algebra is a Hom-Jacobi-Jordan
algebra. Next, we introduce the notion of a representation of a left Hom-pre-Jacobi-Jordan algebra and develop some constructions theorems. The notion of matched pairs and $\mathcal{O}$-operators of such Hom-algebras have also been introduced and interesting results have been obtained. 

Throughout this paper, all vector spaces and algebras are meant over a ground field $\mathbb{K}$ of characteristic 0.
\section{Basic results on Hom-(pre)-Jaobi-Jordan algebras}
This section is devoted to some definitions which are a very useful  for next sections. Some elementary results are also proven.
\begin{definition}
A Hom-module is a pair $(A,\alpha_M)$ consisting of a $\mathbb{K}$-module $A$ and 
a linear self-map $\alpha_A: A\longrightarrow A.$ A morphism 
$f: (A,\alpha_A)\longrightarrow (B,\alpha_B)$ of Hom-modules is a linear map 
 $f: A\longrightarrow B$ such that $f\circ\alpha_A=\alpha_B\circ f.$
\end{definition}
\begin{definition} A Hom-algebra is a triple $(A,\mu,\alpha)$ in which $(A,\alpha)$ is a Hom-module, $\mu: A^{\otimes 2}\longrightarrow A$ is a linear map.
The Hom-algebra $(A,\mu,\alpha)$ is said to be  multiplicative if $\alpha\circ\mu=\mu\circ\alpha^{\otimes 2}.$  A morphism 
$f: (A,\mu_A,\alpha_A)\longrightarrow (B,\mu_B,\alpha_B)$ of Hom-algebras is a morphism of the underlying Hom-modules such that $f\circ\mu_A=\mu_B\circ f^{\otimes 2}.$
\end{definition}
\begin{definition}
 Let $(A,\mu,\alpha)$ be a Hom-algebra and 
 $\lambda\in\mathbb{K}.$  Let $P$ be a linear map satisfying 
 \begin{eqnarray}
  \mu(P(x),P(y))=P(\mu(P(x),y)+\mu(x,P(y))+\lambda\mu(x,y)),\ \forall x,y \in A \label{Rota-Baxt}
 \end{eqnarray}
Then, $P$ is called a Rota-Baxter operator of weight $\lambda$ and $(A,\mu,\alpha, P)$ is called a Rota-Baxter Hom-algebra of weight 
$\lambda.$
\end{definition}
In the sequel, to unify our terminologies by a Rota-Baxter operator (resp. a Rota-Baxter Hom-algebra), we mean a Rota-Baxter operator (resp. a Rota-Baxter Hom-algebra) of weight $\lambda=0.$

\begin{definition}\cite{MAK3}
 A Hom-associative algebra is a multiplicative Hom-algebra $(A,\mu,\alpha)$
satisfying the Hom-associativity condition i.e.,
\begin{eqnarray}
 as_{\alpha}(x,y,z):=\mu(\mu(x,y),\alpha(z))-\mu(\alpha(x),\mu(y,z))=0\label{HAs} \mbox{ for all $x,y,z\in A$}
\end{eqnarray}
 \end{definition}
 \begin{definition}\cite{dy2}
  A Hom-Jordan algebra is a multiplicative Hom-algebra $(A,\ast,\alpha)$ such that the product "$\ast$" is commutative i.e., $x\ast y=y\ast x$ for all $x,y\in A$   and the following so-called Hom-Jordan identity holds
  \begin{eqnarray}
   as_{\alpha}(x\ast x, \alpha(y),\alpha(x))=0 \ \forall x,y\in A
  \end{eqnarray}
 \end{definition}
\section{Hom-Jacobi-Jordan algebras}
\begin{definition}
 A Hom-Jacobi-Jordan algebra is a multiplicative Hom-algebra $(A,\ast,\alpha)$ such that
 \begin{eqnarray}
  && x\ast y=y\ast x \mbox{ ( commutativity)}\nonumber\\
  && J_{\alpha}(x,y,z):=\circlearrowleft_{(x,y,z)}(x\ast y)\ast\alpha(z)=0\label{JJi}
 \end{eqnarray}
 where $\circlearrowleft_{(x,y,z)}$ is the sum over cyclic permutation of $x,y,z$ and 
 $J_{\alpha}$ is called the Hom-Jacobian.
\end{definition}
\begin{remark}
 If $\alpha=Id$ (identity map) in a Hom-Jacobi-Jordan algebra $(A,\ast,\alpha),$ then it reduces to a usual Jacobi-Jordan algebra $(A,\ast).$ It follows that the
category of  Hom-Jacobi-Jordan algebras contains the one of Jacobi-Jordan algebras.
\end{remark}
It is easy to prove the following:
\begin{proposition}
 Let $(A,\ast,\alpha)$ be a Hom-Jacobi-Jordan algebra and $\beta$ be a self-morphism of $(A,\ast,\alpha).$  Then 
 $A_{\beta^n}:=(A,\ast_{\beta^n}:=\beta^n\circ\ast,\beta^n\alpha)$ is a Hom-Jacobi-Jordan algebra for each $n\in\mathbb{N}.$
\end{proposition}
\begin{example}
 (i) Let $(A,\ast)$ be a Jacobi-Jordan algebra and $\alpha$ be a self-morphism of $(A,\ast).$  Then 
 $A_{\alpha}:=(A,\ast_{\alpha}:=\alpha\circ\ast,\alpha)$ is a Hom-Jacobi-Jordan algebra.

 (ii) Consider the $4$-dimensional Jacobi-Jordan algebra $\mathcal{A}:=(A,\ast)$ \cite{dbaf} where non-zero products with respect to a basis $(e_1,e_2,e_3,e_4)$ are given by: $e_1\ast e_1:=e_2;\ e_1\ast e_4=e_4\ast e_1:=e_4.$ Then,  the linear map $\alpha$ defined by
 $\alpha(e_1):=-e_1-e_3;\ \alpha(e_2):=a_{12}e_1+e_2+2e_4;\ \alpha(e_3):=e_1+a_{23}e_2+e_3; \alpha(e_4):=a_{14}e_1-e_2+a_{34}e_3-e_4$ is a self-morphism of $\mathcal{A}$ for all scalars $a_{12}, a_{23}, a_{14}, a_{34}.$ Hence,  $\mathcal{A}_{\alpha}:=(A,\ast_{\alpha}=\alpha\circ\ast,\alpha)$ is a Hom-Jacobi-Jordan algebra where non-zero products are: $e_1\ast_{\alpha} e_1:=a_{12}e_1+e_2+2e_4;\ e_1\ast_{\alpha} e_4=e_4\ast_{\alpha} e_1=:=a_{14}e_1-e_2+a_{34}e_3-e_4.$
 
\end{example}

\begin{proposition}
 Any Hom-Jacobi-Jordan algebra is a Hom-Jordan algebra.
\end{proposition}
\begin{proof}
 Let $(A,\ast,\alpha)$ be a Hom-Jacobi-Jordan algebra. Then, by (\ref{JJi}) with $z=x$ we obtain by the ommutativity of $\ast$
 \begin{eqnarray}
  (x\ast x)\ast\alpha(y)+2\alpha(x)\ast(x\ast y)=0 \mbox{ for all $x,y\in A.$}\label{cHJ1}
 \end{eqnarray}
 Now, if one replaces $x$ by $\alpha(x)$ in 
 (\ref{cHJ1}), we get by the multiplicativity of $\alpha$
 with respect to $\ast$
 \begin{eqnarray}
  \alpha(x\ast x)\ast\alpha(y)+2\alpha^2(x)\ast(\alpha(x)\ast y)=0  
  \mbox{ for all $x,y\in A.$}
  \label{cHJ2}
 \end{eqnarray}
 Replacing $y$ by $x\ast y$ in (\ref{cHJ2}) yields by the multiplicativity of $\alpha$
 with respect to $\ast$
 \begin{eqnarray}
  && \alpha(x\ast x)\ast(\alpha(x)\ast\alpha(y))=-2\alpha^2(x)\ast\Big( \alpha(x)\ast(x\ast y)\Big)=\alpha^2(x)\ast\Big(-2\alpha(x)\ast(x\ast y) \Big)\nonumber\\
  &&=\alpha^2(x)\ast\Big((x\ast x)\ast\alpha(y) \Big) \mbox{ ( by (\ref{cHJ1}) ).}\nonumber
 \end{eqnarray}
The Hom-Jordan identity follows by the commutativity of $\ast.$
\end{proof}
We know that the plus Hom-algebra of any Hom-associative algebra is a Hom-Jordan algebra i.e., any Hom-associative algebra is Hom-Jordan admissible \cite{sa1}. In Hom-Jacobi-Jordan algebras case, this is not true, indeed:
\begin{proposition}
 Let $(A,\cdot,\alpha)$ be a Hom-associative algebra. Then $(A,\star, \alpha)$ is a Hom-Jacobi-Jordan algebra if and only if  
 $2\circlearrowleft_{(x,y,z)}\Big((x\cdot y)\cdot\alpha(z)+(y\cdot x)\cdot\alpha(z) \Big)=0$ for all $x,y,z\in A$ where 
 $x\star y=x\cdot y+y\cdot x.$
\end{proposition}
\begin{proof}
 By straightforward computations for all $x,y,z\in A,$
 \begin{eqnarray}
  &&\circlearrowleft_{(x,y,z)}(x\star y)\star\alpha(z)=(x\cdot y)\cdot\alpha(z)+(y\cdot x)\cdot\alpha(z)+\alpha(z)\cdot(x\cdot y)+\alpha(z)\cdot(y\cdot x)\nonumber\\
  &&+(y\cdot z)\cdot\alpha(x)+(z\cdot y)\cdot\alpha(x)+\alpha(x)\cdot(y\cdot z)+\alpha(x)\cdot(z\cdot y)+(z\cdot x)\cdot\alpha(y)+(x\cdot z)\cdot\alpha(y)\nonumber\\
  &&+\alpha(y)\cdot(z\cdot x)+\alpha(y)\cdot(x\cdot z)=2\circlearrowleft_{(x,y,z)}\Big((x\cdot y)\cdot\alpha(z)+(y\cdot x)\cdot\alpha(z) \Big) \mbox{ ( by Hom-associativity  ). }\nonumber
 \end{eqnarray}
\end{proof}
Now, we give the definition of representations of a Hom-Jacobi-Jordan algebra.
\begin{definition}
 A representation of a Hom-Jacobi-Jordan algebra $(A, \ast, \alpha)$  is a triple $(V,\rho,\phi)$ where $V$ is a vector space, $\phi\in gl(V)$ and $\rho: A\rightarrow gl(V)$ is a linear map
such that the following equalities hold for all $x, y \in A:$
\begin{eqnarray}
\phi\rho(x)=\rho(\alpha(x))\phi; \label{rHJJ1}\\
 \rho(x\ast y)\phi=-\rho(\alpha(x))\rho(y)-\rho(\alpha(y))\rho(x) \mbox{  for all $x,y\in A$ }\label{rHJJ2}
\end{eqnarray}
\end{definition}
\begin{remark}
 If $\alpha=Id_A,\ \phi=Id_V,$ then $(V,\rho)$ is a representation of the Jacobi-Jordan algebra $(A,\cdot)$ \cite{pzu}.
\end{remark}
Hence, we get the following first example.
\begin{example}
 Let $(A, \cdot)$ be a  Jacobi-Jordan algebra and $(V,\rho)$ be  a representation of $(A,\cdot)$ in the usual sense. Then $(V, \rho,Id_V)$ is a representation of the Hom-Jacobi-Jordan algebra $\mathcal{A}:=(A,\cdot, Id_A).$
\end{example}
To give other examples of representations of Hom-Jacobi-Jordan algebras, let prove the following:
\begin{proposition}\label{PEx0}
 Let $\mathcal{A}_1:=(A_1, \ast_1, \alpha_1)$ and $\mathcal{A}_2:=(A_2, \ast_2, \alpha_2)$ be two Hom-Jacobi-Jordan algebras and $f:\mathcal{A}_1\rightarrow \mathcal{A}_2$ be a morphism of Hom-Jacobi-Jordan algebras. Then $A_2^f:=(A_2, \rho,\alpha_2)$ is a representation of $\mathcal{A}_1$ where
 $\rho(a)b:=f(a)\ast_2 b$ for all $(a,b)\in A_1\times A_2.$
\end{proposition}
\begin{proof}
 Fist, we have for all $(x,b)\in A_1\times A_2,$
 \begin{eqnarray}
  \alpha_2(\rho(x)b)=\alpha_2(f(x))\ast_2\alpha_2(b)=f(\alpha_1(x))\ast_2\alpha_2(b)=\rho(\alpha_1(x))\alpha_2(b),\nonumber
 \end{eqnarray}
i.e., $\alpha_2(\rho(x))=\rho(\alpha_1(x))\alpha_2.$ Next, for all $(x,y)\in A_1^{\times 2}$ and $b\in A_2,$ since $f$ is a morphism we have:
\begin{eqnarray}
 &&\rho(x\ast_1 y)\alpha_2(b)=f(x\ast_1 y)\ast_2\alpha_2(b)=(f(x)\ast_2 f(y))\ast_2\alpha_2(b)\nonumber\\
 &&=-(f(y)\ast_2 b)\ast_2\alpha_2f(x)-(b\ast_2 f(x))\ast_2\alpha_2f(y) \mbox{ ( by (\ref{JJi}) )}\nonumber\\
 &&-f(\alpha_1(x))\ast_2(f(y)\ast_2 b)-f(\alpha_1(y))\ast_2(f(x)\ast_2 b) \mbox{ ( commutativity of $\ast_2$ )}\nonumber\\
 &&-\rho(\alpha_1(x))\rho(y)b-\rho(\alpha_1(y))\rho(x)b.\nonumber
\end{eqnarray}
Hence, $\rho(x\ast_1 y)\alpha_2=-\rho(\alpha_1(x))\rho(y)-\rho(\alpha_1(y))\rho(x).$
\end{proof}
Now, using Proposition \ref{PEx0}, we obtain the following example as applications.
\begin{example}
\begin{enumerate}
 \item Let $(A,\ast,\alpha)$ be a Hom-Jacobi-Jordan algebra. Define a left multiplication 
 $ L: A\rightarrow gl(A)$  by $L(x)y:=x\ast y$  for all $x, y \in A.$ Then $(A , L, \alpha)$ is a representation of $(A, \ast, \alpha),$  called a regular representation
 \item Let $(A,\ast,\alpha)$ be a Hom-Jacobi-Jordan algebra and $(B,\alpha)$ be a  Hom-ideal of $(A,\ast,\alpha).$ Then $(B,\alpha)$  inherits a structure of representation of $(A,\ast,\alpha)$ where 
 $\rho(a)b:=a\ast b$ for all $(a,b)\in A\times B.$
\end{enumerate}
\end{example}
The following result can be proven easily.
\begin{proposition}\label{sRHas}
 Let $\mathcal{V}:=(V, \rho,\phi)$ be a representation of a Hom-Jacobi-Jordan algebra $\mathcal{A}:=(A,\ast,\alpha)$ and $\beta$ be a self-morphism of $\mathcal{A}$. Then 
 $\mathcal{V}_{\beta^n}=:(V,\rho_{\beta^n}:=\rho\beta^n,\phi)$ is a representation of $\mathcal{A}$  for each $n\in\mathbb{N}.$ In particular,$\mathcal{V}_{\alpha^n}=:(V,\rho_{\alpha^n}:=\rho\alpha^n,\phi)$ is a representation of $\mathcal{A}$  for each $n\in\mathbb{N}.$  
\end{proposition}
\begin{corollary}
 Let $(A,\ast,\alpha)$ be a Hom-Jacobi-Jordan algebra. For any integer $n\in\mathbb{N},$ define $L^n: A\rightarrow gl(A)$ by 
 \begin{eqnarray}
  L^n(x)y=\alpha^n(x)\ast y \mbox{ for all $x,y\in A.$}\nonumber
 \end{eqnarray}
Then $(A, L^n,\alpha)$ is a representation of the Hom-Jaobi-Jordan algebra 
$(A,\ast,\alpha).$
\end{corollary}
Let us prove the following necessary result.  
\begin{proposition}\label{spHJJ}
 Let $(A,\ast,\alpha)$ be a 
 Hom-Jacobi-Jordan algebra. Then $(V, \rho, \phi)$ is a representation of $(A, \ast,\alpha)$ if and only if the direct sum of vector spaces $A\oplus V,$  turns into a 
Hom-Jacobi-Jordan
algebra with the multiplication and the linear map defined by
\begin{eqnarray}
 (x+u)\diamond (y+v):=x\ast y+(\rho(x)v+\rho(y)u)\label{SdpHJJ1}\\
 (\alpha\oplus\phi)(x+u):=\alpha(x)+\phi(u)\label{sdpHJJ2}
\end{eqnarray}
This Hom-Jacobi-Jordan algebra is called the semi-direct product of $A$ with $V$ and is denoted by $A\ltimes V.$
\end{proposition}
\begin{proof}
It is clear that $\diamond$ is commutative and its multiplicativity with respect to $\alpha\oplus\phi$ is equivalent to Condition (\ref{rHJJ1}) and the multiplicativity of $\ast$ with respect to $\alpha.$ Next, for all $x,y,z\in A,\ u,v,w\in V,$ we have by straightforward computations
\begin{eqnarray}
 &&\circlearrowleft_{(x+u,y+v,z+w)}
 \Big(((x+u)\diamond (y+v))\diamond(\alpha(z)+\phi(w)) \Big)\nonumber\\
 &&=
 \circlearrowleft_{(x+u,y+v,z+w)}
 \Big((x\ast y)\ast\alpha(z)+\rho(x\ast y)\phi(w)+\rho(\alpha(z))\rho(x)v+\rho(\alpha(z))\rho(y)u \Big)\nonumber\\
 &&=\circlearrowleft_{(x,y,z)}(x\ast y)\ast\alpha(z)+\Big(\rho(x\ast y)\phi(w)+\rho(\alpha(x))\rho(y)w+\rho(\alpha(y))\rho(x)w \Big)
 +\Big(\rho(y\ast z)\phi(u)\nonumber\\
 &&+\rho(\alpha(y))\rho(z)u+\rho(\alpha(z))\rho(y)u \Big)
 +\Big(\rho(z\ast x)\phi(v)+\rho(\alpha(z))\rho(x)v+\rho(\alpha(x))\rho(z)v \Big).\nonumber
\end{eqnarray}
Hence, (\ref{JJi}) holds for $A\ltimes V$ if and only if (\ref{rHJJ2}) holds.
\end{proof}
Let $(A, \ast,\alpha)$ be a Hom-Jacobi-Jordan algebra and $(V, \rho,\phi)$ be a representation of  $(A, \ast,\alpha)$ such that $\phi$ is invertible. Define 
$\rho^{\ast}: A\rightarrow gl(V^{\ast})$
by $\langle\rho^{\ast}(x)(\xi), u\rangle=\langle \xi,\rho(x)u\rangle,$ 
$\forall x\in A,\ u\in v, \xi\in V^{\ast}.$
 Note that  in general $\rho^{\ast}$ is not a representation of $A$ (see \cite{sbam} for details). Define 
$\rho^{\ast}: A\rightarrow gl(V^{\ast})$
by
\begin{eqnarray}
 \rho^{\ast}(x)(\xi):=\rho^{\ast}(\alpha(x))((\phi)^{-2})^{\ast}(\xi)) \mbox{ for all $x\in A, \xi\in V^{\ast}.$}
 \label{dualr}
\end{eqnarray}
Now, we can prove 
\begin{theorem}\label{repdual}
 Let $(V,\rho, \phi)$ be a representation of a regular Hom-Jacobi-Jordan algebra 
 $(A, \ast ,\alpha)$ such that $\phi$ is invertible. Then $\rho^{\ast}: A\rightarrow gl(V^{\ast})$
defined above by (\ref{dualr}) is a representation of $(A, \ast ,\alpha)$ on 
$V^{\ast}$ with respect to 
$(\phi^{-1})^{\ast}.$
\end{theorem}
\begin{proof}
 First, by (\ref{rHJJ1}) and (\ref{dualr}), we get for all $x\in A, \xi\in V^{\ast},$
 \begin{eqnarray}
  \rho^{\ast}(\alpha(x))((\phi^{-1})^{\ast}(\xi))=\rho^{\ast}(\alpha^2(x))((\phi^{-3})^{\ast}(\xi))=
  (\phi^{-1})^{\ast}(\rho^{\ast}(\alpha(x))(\phi^{-2})^{\ast}(\xi))=
  (\phi^{-1})^{\ast}(\rho^{\ast}(x)(\xi))\nonumber
 \end{eqnarray}
which implies $\rho^{\ast}(\alpha(x))\circ(\phi^{-1})^{\ast}=(\phi^{-1})^{\ast}\circ\rho^{\ast}(x).$\\

On the other hand, by (\ref{rHJJ2}), for all $x, y \in A,$ $\xi\in V^{\ast}$ and 
$u\in V,$ we have
\begin{eqnarray}
 &&\langle\rho^{\ast}(x\ast y)((\phi^{-1})^{\ast}(\xi)),u\rangle=\langle
 \rho^{\ast}(\alpha(x)\ast\alpha(y))((\phi^{-3})^{\ast}(\xi)),u\rangle=\langle
 (\phi^{-3})^{\ast}(\xi), \rho(\alpha(x)\ast\alpha(y))(u)\rangle\nonumber\\
 &&= \langle(\phi^{-3})^{\ast}(\xi),  -\rho(\alpha^2(x))\rho(\alpha(y))(\phi^{-1}(u))-\rho(\alpha^2(y))\rho(\alpha(x))(\phi^{-1}(u))\rangle \mbox{ ( by (\ref{rHJJ2}) )}\nonumber\\
 &&= \langle(\phi^{-4})^{\ast}(\xi),  -\rho(\alpha^3(x))\rho(\alpha^2(y))(u)-\rho(\alpha^3(y))\rho(\alpha^2(x))(u)\rangle\nonumber\\
 &&=\langle -\rho^{\ast}(\alpha^2(y))\rho^{\ast}(\alpha^3(x))(\phi^{-4})^{\ast}(\xi)-\rho^{\ast}(\alpha^2(x))\rho^{\ast}(\alpha^3(y))(\phi^{-4})^{\ast}(\xi),u\rangle\nonumber\\
 &&=\langle -\rho^{\ast}(\alpha(y))\rho^{\ast}(x)\xi-\rho^{\ast}(\alpha(x))\rho^{\ast}(y)\xi,u\rangle\nonumber
\end{eqnarray}
Hence, $\rho^{\ast}(x\ast y)(\phi^{-1})^{\ast}=-\rho^{\ast}(\alpha(x))\rho^{\ast}(y)-\rho^{\ast}(\alpha(y))\rho^{\ast}(x).$ Therefore, $\rho^{\ast}$ is a representation of $(A, \ast ,\alpha)$ on $V^{\ast}$ with respect to 
$(\phi^{-1})^{\ast}.$
\end{proof}
\begin{corollary}
 Let $(A, \ast ,\alpha)$ be a regular Hom-Jacobi-Jordan algebra. Then $ad^{\ast}: A\rightarrow gl(A^{\ast})$ defined by 
 \begin{eqnarray}
  ad_x^{\ast}\xi:=ad_{\alpha(x)}^{\ast}(\alpha^{-2})^{\ast}(\xi) \mbox{  for all $x\in A, \xi\in A^{\ast}$ }
 \end{eqnarray}
is a representation of the regular Hom-Jacobi-Jordan algebra $(A, \ast ,\alpha)$ on $A^{\ast}$ with respect to $(\alpha^{-1})^{\ast},$ which is called the
coadjoint representation.
\end{corollary}
also, we get:
\begin{corollary}
 Let $(A, \ast ,\alpha)$ be a regular Hom-Jacobi-Jordan algebra. Then there is a natural Hom-Jacobi-Jordan algebra 
 $(A\oplus A^{\ast},\diamond, \alpha\oplus(\alpha^{-1})^{\ast}),$ where the product $\diamond$  is given by 
 \begin{eqnarray}
  &&(x_1+\xi_1)\diamond (x_2+\xi_2):=
  x_1\ast x_2+ad_{x_1}^{\ast}\xi_2+ad_{x_2}^{\ast}\xi_1\nonumber\\
  &&= x_1\ast x_2+ad_{\alpha(x_1)}^{\ast}(\alpha^{-2})^{\ast}(\xi_2)+
  ad_{\alpha(x_2)}^{\ast}(\alpha^{-2})^{\ast}(\xi_1)\mbox{  for all $x_1,x_2\in A, \xi_1,\xi_2\in A^{\ast}.$ }\nonumber
 \end{eqnarray}
\end{corollary}
\begin{definition}
A matched pair of Hom-Jacobi-Jordan algebras denoted by $(A_1,A_2,\rho_1,\rho_2),$ consists of two Hom-Jacobi-Jordan algebras $\mathcal{A}_1:=(A_1,\ast,\alpha_1)$ and 
 $\mathcal{A}_2:=(A_2,\bullet,\alpha_2)$ together with representations $\rho_1: A_1\rightarrow gl(A_2)$ and 
 $\rho_2: A_2\rightarrow gl(A_1)$ with respect to $\alpha_2$ and $\alpha_1$ respectively such that for all $x,y\in A_1,$ $a,b\in A_2.$  the following conditions hold
 \begin{eqnarray}
 && \rho_1(\alpha_1(x))(a\bullet b)+(\rho_1(x)a)\bullet\alpha_2(b)+(\rho_1(x)b)\bullet\alpha_2(a)\nonumber\\
 &&+\rho_1(\rho_2(a)x)\alpha_2(b)+\rho_1(\rho_2(b)x)\alpha_2(a)=0\label{mpHJJ1}\\
 && \rho_2(\alpha_2(a))(x\ast y)+(\rho_2(a)x)\ast\alpha_1(y)+(\rho_2(a)y)\ast\alpha_1(x)\nonumber\\
 &&+\rho_2(\rho_1(x)a)\alpha_1(y)+\rho_2(\rho_1(y)a)\alpha_1(x)=0\label{mpHJJ2}
 \end{eqnarray}
\end{definition}
\begin{theorem}
Let $\mathcal{A}_1:=(A_1,\ast,\alpha_1)$ and 
 $\mathcal{A}_2:=(A_2,\bullet,\alpha_2)$ be Hom-Jacobi-Jordan algebras. Then,
 $(A_1,A_2,\rho_1,\rho_2),$ is a matched pair of Hom-Jacobi-Jordan algebras if and only if  $(A_1\oplus A_2, \diamond, \alpha_1\oplus\alpha_2)$ is a Hom-Jacobi-Jordan algebra where
\begin{eqnarray}
 (x+b)\diamond(y+b)&:=&(x\ast y+\rho_2(a)y+\rho_2(b)x)+(a\bullet b+\rho_1(x)b+\rho_1(y)a) \label{opHJJ1}\\
 (\alpha_1\oplus\alpha_2)(x+a)&:=&\alpha_1(x)+\alpha_2(a)\nonumber
\end{eqnarray}
\end{theorem}
\begin{proof} Fist, the commutativity of $\diamond$ and its multiplicativity with respect to $\alpha_1\oplus\alpha_2$ are equivalent to the multiplicativity of $\ast$ and $\bullet$ with respect to $\alpha_1$ and $\alpha_2$ respectively and Condition (\ref{rHJJ1}). Next, for all $x,y,z\in A_1$ and $a,b,c\in A_2,$ we compute
\begin{eqnarray}
 &&\circlearrowleft_{(x+a,y+b,z+x)}\Big((x+b)\diamond(y+b) \Big)\diamond(\alpha_1\oplus\alpha_2)(z+c)\nonumber\\
 &&=
 \circlearrowleft_{(x+a,y+b,z+x)}\Big( (x\ast y)\ast\alpha(z)+(\rho_2(a)y)\ast\alpha_1(z)+(\rho_2(b)x)\ast\alpha_1(z)+\rho_2(a\bullet b)\alpha_1(z)\nonumber\\
 &&+\rho_2(\rho_1(x)b)\alpha_1(z)+\rho_2(\rho_1(y)a)\alpha_1(z)+\rho_2(\alpha_2(c))(x\ast y)+\rho_2(\alpha_2(c))\rho_2(a)y+\rho_2(\alpha_2(c))\rho_2(b)x\Big)\nonumber\\
 &&+\circlearrowleft_{(x+a,y+b,z+x)}\Big( (a\bullet b)\bullet\alpha(c)+(\rho_1(x)b)\bullet\alpha_2(c)+(\rho_1(y)a)\bullet\alpha_2(c)+\rho_1(x\ast y)\alpha_2(c)\nonumber\\
 &&+\rho_1(\rho_2(a)y)\alpha_2(c)+\rho_1(\rho_2(b)x)\alpha_2(c)+\rho_1(\alpha_1(z))(a\bullet b)+\rho_1(\alpha_1(z))\rho_1(x)b+\rho_1(\alpha_1(z))\rho_1(y)a\Big)\nonumber\\
 &&=\Big( \circlearrowleft_{(x,y,z)}(x\ast y)\ast\alpha_1(z)\Big)+\Big( \rho_2(a\bullet b)\alpha_1(z)+\rho_2(\alpha_2(a))\rho_2(b)z+\rho_2(\alpha_2(b))\rho_2(a)z\Big)\nonumber\\
 &&+\Big( \rho_2(b\bullet c)\alpha_1(x)+\rho_2(\alpha_2(b))\rho_2(c)x+\rho_2(\alpha_2(c))\rho_2(b)x\Big)+
 \Big( \rho_2(c\bullet a)\alpha_1(y)+\rho_2(\alpha_2(c))\rho_2(a)y\nonumber\\
 &&+\rho_2(\alpha_2(a))\rho_2(c)y\Big)+\Big( \rho_2(\alpha_2(c))(x\ast y)+(\rho_2(c)x)\ast\alpha_1(y)+(\rho_2(c)y)\ast\alpha_1(x) +\rho_2(\rho_1(x)c)\alpha_1(y)\nonumber\\
 &&+\rho_2(\rho_1(y)c)\alpha_1(x)\Big)
 +\Big( \rho_2(\alpha_2(b))(z\ast x)+(\rho_2(b)z)\ast\alpha_1(x)+(\rho_2(b)x)\ast\alpha_1(z) +\rho_2(\rho_1(z)b)\alpha_1(x)\nonumber\\
 &&+\rho_2(\rho_1(x)b)\alpha_1(z)\Big)
+\Big( \rho_2(\alpha_2(a))(y\ast z)+(\rho_2(a)y)\ast\alpha_1(z)+(\rho_2(a)z)\ast\alpha_1(y) +\rho_2(\rho_1(y)a)\alpha_1(z)\nonumber\\
 &&+\rho_2(\rho_1(z)a)\alpha_1(y)\Big)+\Big( \circlearrowleft_{(a,b,c)}(a\bullet b)\bullet\alpha_2(c)\Big)+\Big( \rho_1(x\ast y)\alpha_2(c)+\rho_1(\alpha_1(x))\rho_1(y)c+\rho_1(\alpha_1(y))\rho_1(x)c\Big)\nonumber
\end{eqnarray}
\begin{eqnarray}
 &&+\Big( \rho_1(y\ast z)\alpha_2(a)+\rho_1(\alpha_1(y))\rho_1(z)a+\rho_1(\alpha_1(z))\rho_1(y)a\Big)+
 \Big( \rho_1(z\ast x)\alpha_2(b)+\rho_1(\alpha_1(z))\rho_1(x)b\nonumber\\
 &&+\rho_1(\alpha_1(x))\rho_1(z)b\Big)+\Big( \rho_1(\alpha_1(z))(a\bullet b)+(\rho_1(z)a)\bullet\alpha_2(b)+(\rho_1(z)b)\bullet\alpha_2(a) +\rho_1(\rho_2(a)z)\alpha_2(b)\nonumber\\
 &&+\rho_1(\rho_2(b)z)\alpha_2(a)\Big)
 +\Big( \rho_1(\alpha_1(y))(c\bullet a)+(\rho_1(y)c)\bullet\alpha_2(a)+(\rho_1(y)a)\bullet\alpha_2(c) +\rho_1(\rho_2(c)y)\alpha_2(a)\nonumber\\
 &&+\rho_1(\rho_2(a)y)\alpha_2(c)\Big)
+\Big( \rho_1(\alpha_1(x))(b\bullet c)+(\rho_1(x)b)\bullet\alpha_2(c)+(\rho_1(x)c)\bullet\alpha_2(b) +\rho_1(\rho_2(b)x)\alpha_2(c)\nonumber\\
 &&+\rho_1(\rho_2(c)x)\alpha_2(b)\Big).
\end{eqnarray}
Therefore, by (\ref{JJi}) and (\ref{rHJJ2}), we have
\begin{eqnarray}
 &&\circlearrowleft_{(x+a,y+b,z+x)}\Big((x+b)\diamond(y+b) \Big)\diamond(\alpha_1\oplus\alpha_2)(z+c)\nonumber\\
 &&=\Big( \rho_2(\alpha_2(c))(x\ast y)+(\rho_2(c)x)\ast\alpha_1(y)+(\rho_2(c)y)\ast\alpha_1(x) +\rho_2(\rho_1(x)c)\alpha_1(y)+\rho_2(\rho_1(y)c)\alpha_1(x)\Big)\nonumber\\
 &&+\Big( \rho_2(\alpha_2(b))(z\ast x)+(\rho_2(b)z)\ast\alpha_1(x)+(\rho_2(b)x)\ast\alpha_1(z) +\rho_2(\rho_1(z)b)\alpha_1(x)
 +\rho_2(\rho_1(x)b)\alpha_1(z)\Big)\nonumber\\
&&+\Big( \rho_2(\alpha_2(a))(y\ast z)+(\rho_2(a)y)\ast\alpha_1(z)+(\rho_2(a)z)\ast\alpha_1(y) +\rho_2(\rho_1(y)a)\alpha_1(z)
 +\rho_2(\rho_1(z)a)\alpha_1(y)\Big)\nonumber\\
 &&+\Big( \rho_1(\alpha_1(z))(a\bullet b)+(\rho_1(z)a)\bullet\alpha_2(b)+(\rho_1(z)b)\bullet\alpha_2(a) +\rho_1(\rho_2(a)z)\alpha_2(b)
 +\rho_1(\rho_2(b)z)\alpha_2(a)\Big)\nonumber\\
 &&+\Big( \rho_1(\alpha_1(y))(c\bullet a)+(\rho_1(y)c)\bullet\alpha_2(a)+(\rho_1(y)a)\bullet\alpha_2(c) +\rho_1(\rho_2(c)y)\alpha_2(a)
 +\rho_1(\rho_2(a)y)\alpha_2(c)\Big)\nonumber\\
&&+\Big( \rho_1(\alpha_1(x))(b\bullet c)+(\rho_1(x)b)\bullet\alpha_2(c)+(\rho_1(x)c)\bullet\alpha_2(b) +\rho_1(\rho_2(b)x)\alpha_2(c)+\rho_1(\rho_2(c)x)\alpha_2(b)\Big).\nonumber
\end{eqnarray}

Hence, (\ref{JJi}) is satisfied in $A_1\oplus A_2$ if and only if (\ref{mpHJJ1}) and (\ref{mpHJJ2}) hold.
 
\end{proof}
\begin{definition}
 Let $(V,\rho,\phi )$ be a representation of a Hom-Jacobi-Jordan algebra $(A,\ast,\alpha).$ A linear
operator $T : V\rightarrow A$ is called an $\mathcal{O}$-operator of $A$ associated to $\rho$ if it satisfies 
\begin{eqnarray}
 && T\phi=\alpha T \label{rbHJJ1}\\
  && T(u)\ast T(v)=T\Big(\rho(T(u))v+\rho(T(v))u\Big) \mbox{ for all $u,v\in V$} \label{rbHJJ2}
\end{eqnarray}
\end{definition}
Observe  that Rota-Baxter operators on Hom-Jacobi-Jordan algebras are $\mathcal{O}$-operators with respect to the regular representation.
\begin{example}
 Let $(A, \ast, \alpha)$ be a Hom-Jacobi-Jordan algebra and $(V,\rho, \phi)$ be a representation of $(A, \ast, \alpha).$
It is easy to verify that $A\oplus V$ is a representation of $(A, \ast, \alpha)$ under the maps $\rho_{A\oplus V}: A\rightarrow gl(A\oplus V)$ defined by
\begin{eqnarray}
 &&\rho_{A\oplus V}(a)(b+v):=a\ast b+\rho(a)v.\nonumber
\end{eqnarray}
Define the linear map $T: A\oplus V\rightarrow A, a+v\mapsto a.$ Then $T$ is an $\mathcal{O}$-operator on $A$ with
respect to the representation $(A\oplus V,\rho_{A\oplus V},\alpha\oplus\phi).$
\end{example}
Let give another example of $\mathcal{O}$-operators of Hom-Jacobi-Jordan algebras.
%\begin{example}  
%\end{example}
As Hom-associative algebras case \cite{tcsmam}, let give some characterizations of $\mathcal{O}$-operators on Hom-Jacobi-Jordan algebras.
\begin{proposition}
 A linear map $T : V\rightarrow A$ is an $\mathcal{O}$-operator associated to a representation $(V,\rho,\phi)$ of a Hom-Jacobi-Jordan algebra $(A, \ast, \alpha)$ if and only if the graph of $T,$
$$G_r(T):=\{(T(v), v), v\in  V\}$$
is a subalgebra of the semi-direct product algebra $A\ltimes V.$
\end{proposition}
The following result shows that an $\mathcal{O}$-operator can be lifted up the
Rota-Baxter operator.
\begin{proposition}
 Let $(A, \ast, \alpha)$ be a Hom-Jacobi-Jordan algebra, $(V,\rho,\phi)$ be a representation of $A$ and
$T : V\rightarrow A$ be a linear map. Define 
$\widehat{T}\in End(A\oplus V)$ by 
$\widehat{T}(a+v):=Tv.$ Then T is an $\mathcal{O}$-operator associated to $(V,\rho,\phi)$
if and only if $\widehat{T}$ is a Rota-Baxter operator on $A\oplus V.$
\end{proposition}
In order to give another characterization of $\mathcal{O}$-operators, let introduce the following:
\begin{definition}
 Let $(A, \ast, \alpha)$ be a Hom-Jacobi-Jordan algebra. A linear map 
 $N : A\rightarrow A$ is
said to be a Nijenhuis operator if $N\alpha=\alpha N$ and its Nijenhuis torsions vanish, i.e.,
\begin{eqnarray}
 && N(x)\ast N(y)=N(N(x)\ast y + x\ast N(y)-N(x\cdot y)), \mbox{ for all $x, y\in A,$}\nonumber
\end{eqnarray}
Observe that the deformed multiplications
$\ast_N: A\oplus A\rightarrow A$ given by
\begin{eqnarray}
 && x\ast_N y:= N(x)\ast y+x\ast N(y)-N(x\ast y),\nonumber
\end{eqnarray}
gives rise to a new Hom-Jacobi-Jordan multiplication on $A,$ and $N$ becomes a 
morphism from the Hom-Jacobi-Jordan algebra $(A,\ast_N,\alpha)$ to the initial Hom-Jacobi-Jordan algebra $(A, \ast, \alpha).$
\end{definition}
Now, we can esealy check the following result.
\begin{proposition}
 Let $\mathcal{A}:=(A, \ast, \alpha)$ be a Hom-Jacobi-Jordan algebra and $\mathcal{V}:=(V,\rho, \phi)$ be a representation of $(A, \ast, \alpha).$
 A linear map $T: V\rightarrow A$ is an
  $\mathcal{O}$-operator associated to
$\mathcal{V}$ if and only if $N_T:=\left(
\begin{array}{cc}
 0& T\\
 0& 0
\end{array}
\right)
: A\oplus V \rightarrow A\oplus V$ is a Nijenhuis operator on
the Hom-Jordan algebra $A\oplus V.$
\end{proposition}
\section{Hom-pre-Jacobi-Jordan algebras}
In this section, we generalize the notion of left (resp. right) pre-Jacobi-Jordan algebra first introduced in \cite{absb} as left-skew-symmetric (resp. right-skew-symmetric ) algebras  to the Hom
case and study the relationships with Hom-Jacobi-Jordan algebras in terms of $\mathcal{O}$-operators of Hom-Jacobi-Jordan algebras.
\subsection{Definition and basic properties}
Fist let introduce the following:
\begin{definition}
Let $(A,\cdot,\alpha)$ be a Hom-algebra. The anti-Hom-associator of $(A,\cdot,\alpha)$ is the map defined by
\begin{eqnarray}
 as_{\alpha}^t(x,y,z):=(x\cdot y)\cdot\alpha(z)+\alpha(x)\cdot(y\cdot z) 
 \mbox{ for all $x,y,z\in A.$} \label{antiHas}
\end{eqnarray}
A multiplicative Hom-algebra $(A,\cdot,\alpha)$ is said to be 
 an anti-Hom-associative algebra is $as_{\alpha}^t(x,y,z)=0$  for all $x,y,z\in A,$ 
\end{definition}
Clearly, any Hom-associative algebra is an anti-Hom-associative algebra.\\
Now, we give the definition of the main object of this subsection.
\begin{definition}
 A left Hom-pre-Jacobi-Jordan  algebra is a multiplicative Hom-algebra $(A,\cdot,\alpha)$ satisfying
 \begin{eqnarray}
  && as_{\alpha}^t(x,y,z)=-as_{\alpha}^t(y,x,z) \mbox{ for all $x,y,z\in A.$} \label{HpJJi}
 \end{eqnarray}
 i.e., the anti-Hom-associator is left skew-symmetric. Actually, (\ref{HpJJi}) is Equivariant to
 \begin{eqnarray}
  (x\star y)\cdot\alpha(z)=-\alpha(x)\cdot(y\cdot z)-\alpha(y)\cdot(x\cdot z)
  \mbox{ for all $x,y,z\in A.$} \label{HpJJi2}
 \end{eqnarray}
where $x\star y=x\cdot y=y\cdot x$ for all 
$x,y\in A.$
\end{definition}
If the anti-Hom-associator is right skew-symmetric i.e., 
\begin{eqnarray}
 as_{\alpha}^t(x,y,z)=-as_{\alpha}^t(x,z,y)\label{HpJJir}
\end{eqnarray}
or equivalently
\begin{eqnarray}
 \alpha(x)\cdot(y\star z)=-(x\cdot y)\cdot\alpha(z)-(x\cdot z)\cdot\alpha(y) \mbox{ for all $x,y,z\in A,$}\nonumber
\end{eqnarray}
then, the multiplicative Hom-algebra is said to be a right Hom-pre-Jacobi-Jordan algebra.

If $\alpha=Id$( identity map) in (\ref{HpJJi} (resp.  (\ref{HpJJir})), we obtain the identity defining the so-called left (resp. right) pre-Jacobi-Jordan algebra. Hence, any pre-Jacobi-Jordan algebra is a pre-Hom-Jacobi-Jordan algebra with  $Id$ as twisting map.
\begin{remark}
 (i) Any anti-associative algebra is a left and right Hom-pre-Jacobi-Jordan algebra.\\
 (ii) Observe that if $(A, \cdot,\alpha)$ is a left Hom-pre-Jacobi-Jordan algebra,
then, the Hom-algebra defined on the same vector space A with "opposite" 
multiplication $x\perp y:=y\cdot x$  is a right Hom-pre-Jacobi-Jordan algebra and vice-versa. Hence, all the statements for left Hom-pre-Jacobi-Jordan algebras have their corresponding statements
for left Hom-pre-Jacobi-Jordan algebras. Thus, we will only consider the left Hom-pre-Jacobi-Jordan algebra case in this paper that we often call Hom-pre-Jacobi-Jordan algebra for short.
\end{remark}
It is easy to prove the following.
\begin{proposition}
 Let $\mathcal{A}:=(A,\mu,\alpha)$ be a Hom-pre-Jacobi-Jordan algebra and $\beta$ be a morphism of $\mathcal{A}.$ Then, $\mathcal{A}_{\beta^n}:=(A,\mu_{\beta^n}:=\beta^n\mu,\beta^n\alpha)$ 
 is a Hom-pre-Jacobi-Jordan algebra for each $n\in\mathbb{N}.$ In particular, $\mathcal{A}_{\alpha^n}:=(A,\mu_{\alpha^n}:=\alpha^n\mu,\alpha^{n+1})$ 
is a Hom-pre-Jacobi-Jordan algebra for each $n\in\mathbb{N}.$
 \end{proposition}
\begin{example}
 Let $\mathcal{A}:=(A,\mu)$ be a pre-Jacobi-Jordan algebra and $\beta$ be a morphism of $\mathcal{A}.$ Then, $\mathcal{A}_{\beta^n}:=(A,\mu_{\beta^n}:=\beta^n\mu,\beta^n)$ 
 is a Hom-pre-Jacobi-Jordan algebra for each $n\in\mathbb{N}.$
\end{example}
\begin{proposition}\label{lpHJJHJJ}
 Let $(A,\cdot , \alpha)$ be a left Hom-pre-Jacobi-Jordan algebra. Then the product given by
\begin{eqnarray}
 x\star y=x\cdot y+y\cdot x \label{asHJJ}
\end{eqnarray}
defines a Hom-Jacobi-Jordan algebra structure on $A,$ which is called the associated (or sub-adjacent) Hom-Jacobi-Jordan algebra
of $(A,\cdot, \alpha)$ denoted by $A^C$ and $(A,\cdot, \alpha)$ is also called a compatible left Hom-pre-Jacobi-Jordan algebra structure on the Hom-Jacobi-Jordan algebra $A^C=(A,\star,\alpha).$
\end{proposition}
\begin{proof}
 For all $x,y,z\in A,$ we prove (\ref{JJi}) as follows
 \begin{eqnarray}
  &&J_{\alpha}(x,y,z)=\circlearrowleft_{(x,y,z)}(x\star y)\star\alpha(z)\nonumber\\
  &&=\circlearrowleft_{(x,y,z)}
  \Big((x\cdot y)\cdot\alpha(z)
  +(y\cdot x)\cdot\alpha(z)+\alpha(z)\cdot(x\cdot y)+\alpha(z)\cdot(y\cdot x) \Big)\nonumber\\
  &&=\circlearrowleft_{(x,y,z)}\Big(as_{\alpha}^t(x,y,z)+as_{\alpha}^t(y,x,z) \Big)=0 \mbox{ ( by (\ref{HpJJi}) ).}\nonumber
 \end{eqnarray}
\end{proof}
As consequence, we get
\begin{proposition}
 Let $(A,\cdot, \alpha)$ be a Hom-algebra. Then $(A, \cdot, \alpha)$ is a left Hom-pre-Jacobi-Jordan algebra if and
only if $(A,\star , \alpha)$ defined by Eq. (\ref{asHJJ}) is a Hom-Jacobi-Jordan algebra and $(A, L, \alpha)$ is a representation of
$(A,\star, \alpha),$ where $L$ denotes the left multiplication operator on $A.$
\end{proposition}
\begin{proof}
 Straightforward.
\end{proof}
\begin{proposition}\label{HJJpJJ}
 Let $(A,\ast, \alpha)$ be a Hom-Jacobi-Jordan algebra and $(V, \rho, \phi)$ be a representation. If $T$ is
an $\mathcal{O}$-operator associated to $\rho$, then $(V, \cdot, \phi)$ is a left Hom-pre-Jacobi-Jordan algebra, where
\begin{eqnarray}
 u\cdot v:=\rho(T(u))v \mbox{ for $u,v\in V$.} \label{revasHJJ1}
\end{eqnarray}
Therefore there exists an associated Hom-Jacobi-Jordan algebra structure on $V$ given by Eq. (\ref{asHJJ}) and $T$
is a homomorphism of Hom-Jacobi-Jordan algebras. Moreover, $T(V):=\{T(v)|v\in V\} \subset A$ is a Hom-Jacobi-Jordan
subalgebra of $(A, \ast, \alpha)$ and there is an induced left Hom-pre-Jacobi-Jordan algebra structure on $T(V)$ given
by
\begin{eqnarray}
 T(u)\bullet T(v):=T(u\ast v) \mbox{ for $u,v\in V$.} \label{revasHJJ2}
\end{eqnarray}
The corresponding associated Hom-Jacobi-Jordan algebra structure on $T(V)$ given by Eq. (\ref{asHJJ}) is just a
Hom-Jacobi-Jordan subalgebra of $(A, \ast, \alpha)$ and $T$ is a homomorphism of left Hom-pre-Jacobi-Jordan algebras.
\end{proposition}
\begin{proof}
 Let $u, v, w\in V$ and put  
 $u\star v=u\cdot v + v\cdot u.$ Note
first that $T(u\star v) = T(u)\ast T(v).$ Then using (\ref{rbHJJ1}), we compute 
(\ref{HpJJi2}) as follows
\begin{eqnarray}
 &&(u\star v)\cdot\phi(w)=\rho(T(u)\ast T(v))\phi(w)-\rho(T\phi(u))\rho(tv)w-\rho(T\phi(v))\rho(tu)w\nonumber\\
 &&=-\phi(u)\cdot(v\cdot w)-\phi(v)\cdot(u\cdot w).\nonumber
\end{eqnarray}
Therefore, $(V, \cdot, \phi)$ is a left Hom-pre-Jacobi-Jordan algebra. The other conclusions follow immediately.
\end{proof}
An obvious consequence of Proposition \ref{HJJpJJ} is the following construction of a left Hom-pre-Jacobi-Jordan algebra in terms of a Rota-Baxter operator (of weight zero) of a Hom-Jacobi-Jordan algebra.
\begin{corollary}
 Let $(A,\ast, \alpha)$ be a Hom-Jacobi-Jordan algebra and $P$ be a Rota-Baxter operator (of weight
zero) on $A.$ Then there is a left Hom-pre-Jacobi-Jordan algebra structure on $A$ given by
\begin{eqnarray}
 x\cdot y:=P(x)\ast y \mbox{ for all $x,y\in A.$}
\end{eqnarray}
\end{corollary}
\begin{proof}
 Straightforward.
\end{proof}
%\begin{example}
%\end{example}
\begin{corollary}
 Let $(A,\ast, \alpha)$ be a Hom-Jacobi-Jordan algebra. Then there exists a compatible left Hom-pre-Jacobi-
Jordan algebra structure on $A$ if and only if there exists an invertible $\mathcal{O}$-operator of $(A, \ast, \alpha).$
\end{corollary}
\begin{proof}
 Let $(A,\cdot, \alpha)$ be a left Hom-pre-Jacobi-Jordan algebra and $(A,\star, \alpha)$ be the associated Hom-Jacobi-Jordan
algebra. Then the identity map $id : A\rightarrow A$ is an invertible $\mathcal{O}$-operator of $(A,\star,\alpha)$ associated to $(A, ad, \alpha).$

Conversely, suppose that there exists an invertible $\mathcal{O}$-operator $T$ of $(A, \ast, \alpha)$ associated
to a representation $(V, \rho, \phi),$ then by Proposition \ref{HJJpJJ}, there is a left Hom-pre-Jacobi-Jordan algebra
structure on $T(V)=A$ given by
$$T(u)\cdot T(v)=T(\rho(T(u))v),\mbox{ for all $u, v \in V.$}$$
If we set $T(u)=x$ and $T(v)=y,$ then we obtain
$$ x\cdot y=T(\rho(x)T^{-1}(y)), \mbox{ for all $x, y \in A.$}.$$
It is a compatible left Hom-pre-Jacobi-Jordan algebra structure on $(A,\ast, \alpha).$ Indeed,
\begin{eqnarray}
 &&x\cdot y+y\cdot x=T(\rho(x)T^{-1}(y)+\rho(y)T^{-1}(x))\nonumber\\
 &&= T(T^{-1}(x))\ast T(T^{-1}(y))=x\ast y.\nonumber
\end{eqnarray}
\end{proof}
%\begin{example}
%\end{example}
\begin{proposition}
 Let $T : V\rightarrow A$ be an $\mathcal{O}$-operator on the Hom-Jacobi-Jordan algebra $( A, \ast,\alpha)$ with respect to the representation $(V,\rho, \phi).$
Let us define a map $\rho_T : V\rightarrow  gl(A)$ given by
 \begin{eqnarray}
  \rho_T(u)x:=T(u)\ast x-T(\rho(x)u) \mbox{ for all $(u,x)\in V\times A $}\nonumber
 \end{eqnarray}
Then, the triplet $(A,\rho_T,\alpha)$ is a representation of the sub-adjacent Hom-Jacobi-Jordan algebra $V^c=(V,\star,\phi)$ associated with the left Hom-pre-Jacobi-Jordan algebra 
$(V,\cdot,\phi)$ defined in Proposition \ref{HJJpJJ}.
\end{proposition}
\begin{proof}
 First, pick $(u,x)\in V\times A.$ Then, the multiplicativity of $\alpha$ with respect to $\ast$, conditions (\ref{rbHJJ1}) and (\ref{rHJJ1}) in $(V,\rho,\phi)$ give rise  to (\ref{rHJJ1}) for $(A,\rho_T,\alpha)$ as follows
 \begin{eqnarray}
 && \alpha(\rho_T(u)x)=\alpha T(u)\ast\alpha(x)-\alpha T(\rho(x)u)=T\phi(u)\ast\alpha(x)-T(\rho(\alpha(x))\phi(u))=
 \rho_T(\phi(u))\alpha(x).\nonumber
  \end{eqnarray}
Next, let $u,v\in V,\ x\in A.$ Recall that $u\star v=\rho(Tu)v+\rho(Tv)u$ and $T(u\star v)=T(u)\ast T(v).$ Then, by  straightforward  computaions, we have
\begin{eqnarray}
 &&\rho_T(u\star v)\alpha(x)=(T(u)\ast T(v))\ast\alpha(x)-T\Big(\rho(\alpha(x))\rho(Tu)v\Big)-T\Big(\rho(\alpha(x))\rho(Tv)u\Big).\nonumber
\end{eqnarray}
Also, we compute
\begin{eqnarray}
 &&\rho_T(\phi(u))\rho_T(v)x=T\phi(u)\ast(Tv\ast x)-T\phi(u)\ast T(\rho(x)v)-
 T\Big(\rho(Tv\ast x )\phi(u)\Big)\nonumber\\
 &&+T\Big( \rho(T(\rho(x)v))\phi(u)\Big)=
 T\phi(u)\ast(Tv\ast x)-T\Big(\rho(T\phi(u))\rho(x)v+\rho(T(\rho(x)v))\phi(u) \Big)\nonumber\\
 &&-
 T\Big(\rho(T\phi(v))\rho(x)u+\rho(\alpha(x))\rho(Tv)u \Big)+T\Big( \rho(T(\rho(x)v))\phi(u)\Big)
 \mbox{ (by (\ref{rbHJJ2}), (\ref{rbHJJ1}) and (\ref{rHJJ2})  ) }\nonumber\\
 &&=
 T\phi(u)\ast(Tv\ast x)-T\Big(\rho(T\phi(u))\rho(x)v\Big)+
 T\Big(\rho(T\phi(v))\rho(x)u\Big)+T\Big(\rho(\alpha(x))\rho(Tv)u \Big).\nonumber
\end{eqnarray}
Switching $u$ and $v$ in the above equation, we come to
\begin{eqnarray}
 \rho_T(\phi(v))\rho_T(u)x=
 T\phi(v)\ast(Tu\ast x)-T\Big(\rho(T\phi(v))\rho(x)u\Big)+
 T\Big(\rho(T\phi(u))\rho(x)v\Big)+T\Big(\rho(\alpha(x))\rho(Tu)v \Big).\nonumber
\end{eqnarray}
It follows by (\ref{JJi}) that
\begin{eqnarray}
 -\rho_T(\phi(u))\rho_T(v)x-\rho_T(\phi(v))\rho_T(u)x=\rho_T(u\star v)\alpha(x),\nonumber
\end{eqnarray}
i.e., (\ref{rHJJ2}) holds in $(A,\rho_T,\alpha).$
\end{proof}
\subsection{Representations and $\mathcal{O}$-operators} This subsection is devoted to the study of representations and 
$\mathcal{O}$-operators of Hom-pre-Jacobi-Jordan algebras. We first give the definition of representations of pre-Jacobi-Jordan algebras and then, generalize this definition in the Hom-case. As mentioned, let first introduce the notion of a representation of a left Jacobi-Jordan algebra.
\begin{definition}\label{rpHJJcl}
 A representation of a left pre-Jacobi-Jordan algebra $(A,\cdot)$ on a vector space $V$ consists of a pair $(\rho, \lambda),$ where $\rho,\lambda: A\rightarrow gl(V )$  are linear maps satisfying:
\begin{eqnarray}
&&\rho(x\star y)=-\rho(x)\rho(y)-\rho(y)\rho(x)\label{crpJJ1}\\
 &&\lambda(y)\lambda(x)+\lambda(x\cdot y)=-\lambda(y)\rho(x)-\rho(x)\lambda(y) \label{crppHJJ2}
\end{eqnarray}
for all $x,y\in A$ where $x\star y:=x\cdot y+y\cdot x.$
\end{definition}
Observe that, Condition (\ref{crpJJ1}) means that $(V,\rho)$ is a representation of the subadjacent Jacobi-Jordan of $(A,\cdot).$ One can easily prove that $(V,\rho,\lambda)$  is a representation of a left pre-Jacobi-Jordan algebra $(A,\cdot)$ if and only if the direct sum $A\oplus V$ of vector space turns into a left pre-Jacobi-Jordan algebra under the product
\begin{eqnarray}
(x+u)\diamond (y+v):=x\cdot y+(\rho(x)v+\lambda(y)u) \mbox{ for all $x,y\in A$ 
 and $u,v\in V.$}\nonumber
\end{eqnarray}
This left pre-Jacobi-Jordan algebra is called a semi-direct product of $A$ and $V$  denoted  by $A\ltimes V.$
\begin{remark}
 Our definition of representations of left pre-Jacobi-Jordan algebras given above is different of those given in \cite{cedh}. 
\end{remark}

Now, we can extend the notions of representations of pre-Jacobi-Jordan algebras in the Hom-case as follows:
\begin{definition}
 A representation of a left Hom-pre-Jacobi-Jordan algebra $(A,\cdot, \alpha)$ on a vector space $V$ with
respect to $\phi\in gl(V)$ consists of a pair $(\rho, \lambda),$ where $\rho : A\rightarrow gl(V )$ is a representation of the sub-adjacent Hom-Jacobi-Jordan algebra $A^C$ on $V$ with respect to $\phi\in gl(V ),$ and $\lambda: A\rightarrow gl(V )$ is a linear map satisfying:
\begin{eqnarray}
 &&\phi\lambda(x)=\lambda(\alpha(x))\phi\label{rppHJJ1}\\
 &&\lambda(\alpha(y))\lambda(x)+\lambda(x\cdot y)\phi=-\lambda(\alpha(y))\rho(x)-\rho(\alpha(x))\lambda(y) \label{rppHJJ2}
\end{eqnarray}
\end{definition}
\begin{remark}
 If $\alpha=Id_A,\ \phi=Id_V,$ then $(V,\rho,\lambda)$ is a representation of the pre-Jacobi-Jordan algebra $(A,\cdot)$ (see Definition \ref{rpHJJcl} above).
\end{remark}
Hence, we obtain
\begin{example}
 Let $(A, \cdot)$ be a  left pre-Jacobi-Jordan algebra and $(V,\rho,\lambda)$ be  a representation of $(A,\cdot)$ in the usual sense. Then $(V, \rho,\lambda, Id_V)$ is a representation of the left Hom-pre-Jacobi-Jordan algebra $\mathcal{A}:=(A,\cdot, Id_A).$
\end{example}
To give other examples of representations of Hom-Jacobi-Jordan algebras, let prove the following:
\begin{proposition}\label{PEx1}
 Let $\mathcal{A}_1:=(A_1, \cdot, \alpha_1)$ and $\mathcal{A}_2:=(A_2, \intercal, \alpha_2)$ be two left Hom-pre-Jacobi-Jordan algebras and $f:\mathcal{A}_1\rightarrow \mathcal{A}_2$ be a morphism of left pre-Hom-Jacobi-Jordan algebras. Then $A_2^f:=(A_2, \rho,\lambda,\alpha_2)$ is a representation of $\mathcal{A}_1$ where
 $\rho(a)b:=f(a)\intercal b$  and $\lambda(a)b=b\intercal f(a)$ for all $(a,b)\in A_1\times A_2.$
\end{proposition}
\begin{proof}  
Fist, it is clear that $\alpha_2\rho(x)=\rho(\alpha_1(x))\alpha_2$ and $\alpha_2\lambda(x)=\lambda(\alpha_1(x))\alpha_2$ for all $x\in A_1.$ Next, 
  set $x\star y:=x\cdot y+y\cdot x$ and 
 $b\circledast c=b\intercal c+c\intercal b$
  for all $x,y\in A_1$ and $b, c\in A_2.$ Then, one can show that $f(x\star y)=f(x)\circledast f(y).$ Hence, since $f$ is a morphism, we have:
  \begin{eqnarray}
  && \rho(x\star y)\alpha_2(b)=f(x\star y)\alpha_2(b)=
   (f(x)\circledast f(y))\intercal\alpha_2(b)\nonumber\\
   &&=(f(x)\intercal f(y))\intercal\alpha_2(b)+(f(x)\intercal f(y))\intercal\alpha_2(b)\nonumber\\
   &&=-\alpha_2f(x)\intercal(f(y)\intercal b)-\alpha_2f(y)\intercal(f(x)\intercal b) \mbox{ ( by  Condition \ref{HpJJi2} in $A_2$  ) }\nonumber\\
   &&=-\rho(\alpha_1(x))\rho(y)b-\rho(\alpha_1(y))\rho(x)b,\nonumber
  \end{eqnarray}
  i.e., (\ref{rHJJ2}) is satisfied. Finally, using again $f$ is a morphism, we obtain by straightforward computations
  \begin{eqnarray}
   &&\lambda(\alpha_1(y))\lambda(x)b+\lambda(x\cdot y)\alpha_2(b)\nonumber\\
   &&=(b\intercal f(x))\intercal\alpha_2f(x)+\alpha_2(b)\intercal (f(x)\intercal f(y))\nonumber\\
   &&=-(f(x)\intercal b)\intercal\alpha_2f(y)-\alpha_2f(x)\intercal (b\intercal f(y)) \mbox{ ( by  Condition \ref{HpJJi} in $A_2$  ) }    \nonumber\\
   &&-\lambda(\alpha_1(y))\rho(x)b-\rho(\alpha_1(x))\lambda(y)v. \nonumber
  \end{eqnarray}
Hence, we get also (\ref{rppHJJ2}) and the conclusion follows.
\end{proof}
Now, using Proposition \ref{PEx1}, we obtain the following examples as applications.
\begin{example}
\begin{enumerate}
 \item Let $(A,\cdot,\alpha)$ be a left Hom-pe-Jacobi-Jordan algebra. Define a left and right multiplications 
 $ L, R: A\rightarrow gl(A)$  by $L(x)y:=x\cdot y$ and $R(x)y:=y\cdot x$  for all $x, y \in A.$ Then $(A , L, R, \alpha)$ is a representation of $(A, \ast, \alpha),$  called a regular representation
 \item Let $(A,\cdot,\alpha)$ be a left Hom-pre-Jacobi-Jordan algebra and $(B,\alpha)$ be a  two-sided-Hom-ideal of $(A,\cdot,\alpha).$ Then $(B,\alpha)$  inherits a structure of representation of $(A,\cdot,\alpha)$ where 
 $\rho(a)b:=a\cdot b$  and $\lambda(a)(b)=b\cdot a$ for all $(a,b)\in A\times B.$
\end{enumerate}
\end{example}
Similarly, the following result can be proven.
\begin{proposition}\label{sRHas}
 Let $\mathcal{V}:=(V, \rho,\lambda,\phi)$ be a representation of a left Hom-pre-Jacobi-Jordan algebra $\mathcal{A}:=(A,\cdot,\alpha)$ and $\beta$ be a self-morphism of $\mathcal{A}$. Then 
 $\mathcal{V}_{\beta^n}=:(V,\rho_{\beta^n}:=\rho\beta^n,\lambda_{\beta^n}:=\lambda\beta^n, \phi)$ is a representation of $\mathcal{A}$  for each $n\in\mathbb{N}.$ In particular, 
 $\mathcal{V}_{\alpha^n}=:(V,\rho_{\alpha^n}:=\rho\alpha^n,\lambda_{\alpha^n}:=\lambda\alpha^n, \phi)$ is a representation of $\mathcal{A}$  for each $n\in\mathbb{N}.$
\end{proposition}
\begin{corollary}
 Let $(A,\cdot,\alpha)$ be a left Hom-pre-Jacobi-Jordan algebra. For any integer $n\in\mathbb{N},$ define $L^n, R^n: A\rightarrow gl(A)$ by 
 \begin{eqnarray}
  L^n(x)y=\alpha^n(x)\cdot y,R^n(x)y=y\cdot\alpha^n(x) \mbox{ for all $x,y\in A.$}\nonumber
 \end{eqnarray}
Then $(A, L^n, R^n,\alpha)$ is a representation of the left Hom-pre-Jacobi-Jordan algebra 
$(A,\ast,\alpha).$
\end{corollary}
\begin{proposition}\label{spHpJJ}
 Let $(A,\cdot, \alpha)$ be a left Hom-pre-Jacobi-Jordan algebra, $V$ be a vector space, $\rho,\lambda: A\rightarrow gl(V)$
be linear maps and $\phi\in gl(V).$ Then $(V, \rho,\lambda,\phi)$ is a representation of $A$ if and only if the direct sum
 $A\oplus V$ (as vector space) turns into a left Hom-pre-Jacobi-Jordan algebra (the semidirect sum) by defining the
multiplication in $A\oplus V$ as
\begin{eqnarray}
 (x+u)\diamond (y+v):=x\cdot y+(\rho(x)v+\lambda(y)u) \mbox{ for all $x,y\in A$ 
 and $u,v\in V.$}
\end{eqnarray}
We denote it by $A\ltimes V.$
\end{proposition}
\begin{proof}
 The multiplicativity of $\diamond$ with respect to $\alpha\oplus\phi$ is equivalent to conditions (\ref{rppHJJ1}), (\ref{rppHJJ1}) and the multiplicativity of $\cdot$ with respect to $\alpha.$ Next, for all $x,y,z\in A,\ u,v,w\in V,$ we have by straightforward computations
 \begin{eqnarray}
  &&as_{\alpha\oplus\phi}^t(x+u,y+v,z+w)=
  (x\cdot y)\cdot\alpha(z)+\rho(x\cdot y)\phi(w)+\lambda(\alpha(z))\rho(x)v+\lambda(\alpha(z))\lambda(y)u\nonumber\\
  &&+\alpha(x)\cdot(y\cdot z)+\rho(\alpha(x))\rho(y)w+\rho(\alpha(x))\lambda(z)v+\lambda(y\cdot z)\phi(u)=as_{\alpha}^t(x,y,z)+\Big(\rho(x\cdot y)\phi(w)\nonumber\\
  &&+\rho(\alpha(x))\rho(y)w \Big)+\Big(\lambda(\alpha(z))\rho(x)v+\rho(\alpha(x))\lambda(z)v \Big)+
  \Big( \lambda(\alpha(z))\lambda(y)u+\lambda(y\cdot z)\phi(u)\Big)\nonumber
 \end{eqnarray}
Switching $x+u$ and $y+v$ in the expression above of  $as_{\alpha\oplus\phi}^t(x+u,y+v,z+w),$ we get
\begin{eqnarray}
 &&as_{\alpha\oplus\phi}^t(y+v,x+u,z+w)=
 as_{\alpha}^t(x,y,z)+\Big(\rho(x\cdot y)\phi(w)+\rho(\alpha(x))\rho(y)w \Big)+\Big(\lambda(\alpha(z))\rho(x)v\nonumber\\&&+\rho(\alpha(x))\lambda(z)v \Big)+
  \Big( \lambda(\alpha(z))\lambda(y)u+\lambda(y\cdot z)\phi(u)\Big).\nonumber
\end{eqnarray}
Hence, (\ref{HpJJi}) holds for $A\oplus V$ if and only if (\ref{rHJJ2}) (with $\ast=\star$) and (\ref{rppHJJ2}) hold.
\end{proof}
\begin{proposition}\label{sumrl}
 Let $(V, \rho,\lambda,\phi)$ is a representation of a left Hom-pre-Jacobi-Jordan algebra $(A,\cdot, \alpha)$ and $(A,\star, \alpha)$
be its associated Hom-Jacobi-Jordan algebra. Then
 $(V,\rho+\lambda,\phi)$ is a representation of $(A,\star,\alpha),$
\end{proposition}
\begin{proof}
 By Proposition \ref{spHpJJ}, $A\ltimes V$
is a left Hom-pre-Jacobi-Jordan algebra. Consider its associated Hom-Jacobi-Jordan algebra $(A\oplus V,\tilde{\diamond}, \alpha+\phi),$
we have
\begin{eqnarray}
 &&(x+u)\tilde{\diamond} (y+v)=(x+u)\diamond (y+v)+
 (y+v)\diamond (x+u)\nonumber\\
 &&=x\cdot y+(\rho(x)v+\lambda(y)u)+
 y\cdot x+(\rho(y)u+\lambda(x)v)\nonumber\\
 &&=x\star y+\Big((\rho+\lambda)(x)v+(\rho+\lambda)(y)u \Big)\nonumber
\end{eqnarray}
Thanks to Proposition \ref{spHJJ}, we deduce that $(V,\rho+\lambda,\phi)$ is a representation of $(A,\ast,\alpha).$
\end{proof}
As Hom-Jacobi-Jordan algebras case,
let $(A, \cdot,\alpha)$ be a regular left Hom-pre-Jacobi-Jordan algebra and $(V, \rho,\lambda,\phi)$ be a representation of  $(A, \cdot,\alpha)$ such that $\phi$ is invertible. Define 
$\rho^{\ast}, \lambda^{\ast}: A\rightarrow gl(V^{\ast})$
by $\langle\rho^{\ast}(x)(\xi), u\rangle=\langle \xi,\rho(x)u\rangle,$ and 
$\langle\lambda^{\ast}(x)(\xi), u\rangle=\langle \xi,\lambda(x)u\rangle,$
$\forall x\in A,\ u\in v, \xi\in V^{\ast}.$
 Define 
$\rho^{\ast}, \lambda^{\ast}: A\rightarrow gl(V^{\ast})$
by
\begin{eqnarray}
 &&\rho^{\ast}(x)(\xi):=\rho^{\ast}(\alpha(x))((\phi)^{-2})^{\ast}(\xi)) \mbox{ for all $x\in A, \xi\in V^{\ast}.$}
 \label{drpHJJ1}\\
&& \lambda^{\ast}(x)(\xi):=\lambda^{\ast}(\alpha(x))((\phi)^{-2})^{\ast}(\xi)) \mbox{ for all $x\in A, \xi\in V^{\ast}.$}
 \label{drpHJJ2}
\end{eqnarray}
Now, we can prove 
\begin{theorem}\label{drplHJJ}
 Let $(V,\rho,\lambda, \phi)$ be a representation of a regular left Hom-pre-Jacobi-Jordan algebra 
 $(A, \cdot ,\alpha)$ such that $\phi$ is invertible. Then $(V, \rho^{\ast}+\lambda^{\ast}, -\lambda^{\ast},(\phi^{-1})^{\ast})$ is a representation of $(A, \cdot ,\alpha)$  
called the dual representation of $(A, \cdot ,\alpha).$
\end{theorem}
\begin{proof}
First, since $(V,\rho,\lambda, \phi)$ is a representation of a left Hom-pre-Jacobi-Jordan algebra $(A, \cdot ,\alpha),$ thanks to Proposition \ref{sumrl},
we get that $(V,\rho+\lambda, \phi)$ is a representation of the sub-adjacent Hom-Jacobi-Jordan $A^C=(A,\star, \alpha).$ Hence,
by Theorem \ref{repdual}, we deduce that
$(V^{\ast},(\rho+\lambda)^{\ast}=\rho^{\ast}+\lambda^{\ast}, (\phi^{-1})^{\ast})$ is a representation of the Hom-Jacobi-Jordan algebra $A^C.$
Next, by (\ref{rppHJJ1}) and (\ref{drpHJJ2}), we get for all $x\in A, \xi\in V^{\ast},$
 \begin{eqnarray}
  -\lambda^{\ast}(\alpha(x))((\phi^{-1})^{\ast}(\xi))=-\lambda^{\ast}(\alpha^2(x))((\phi^{-3})^{\ast}(\xi))=
  (\phi^{-1})^{\ast}(-\lambda^{\ast}(\alpha(x))(\phi^{-2})^{\ast}(\xi))=
  (\phi^{-1})^{\ast}(-\lambda^{\ast}(x)(\xi))\nonumber
 \end{eqnarray}
which implies $-\lambda^{\ast}(\alpha(x))\circ(\phi^{-1})^{\ast}=(\phi^{-1})^{\ast}\circ(-\lambda^{\ast}(x)).$\\

Finally, using (\ref{drpHJJ1} and \ref{drpHJJ2}), we compute for all $x, y \in A,$ $\xi\in V^{\ast}$ and 
$u\in V:$
\begin{eqnarray}
 &&\langle\lambda^{\ast}(x\cdot y)(\phi^{-1})^{\ast}(\xi),u\rangle=\langle\lambda^{\ast}(\alpha(x)\cdot\alpha(y))(\phi^{-3})^{\ast}(\xi),u\rangle=\langle
 (\phi^{-3})^{\ast}(\xi),\lambda(\alpha(x)\cdot\alpha(y))\rangle\nonumber\\
 &&= \langle(\phi^{-3})^{\ast}(\xi), -\lambda(\alpha^2(y))\lambda(\alpha(x))\phi^{-1}(u) -\lambda(\alpha^2(y))\rho(\alpha(x))\phi^{-1}(u)-\rho(\alpha^2(x))\lambda(\alpha(y))\phi^{-1}(u)\rangle \mbox{ ( by (\ref{rppHJJ2}) )}\nonumber\\
 &&= \langle(\phi^{-4})^{\ast}(\xi), -\lambda(\alpha^3(y))\lambda(\alpha^2(x))u -\lambda(\alpha^3(y))\rho(\alpha^2(x))u-\rho(\alpha^3(x))\lambda(\alpha^2(y))u\rangle\nonumber\\
 &&=\langle-\lambda^{\ast}(\alpha^2(x))\lambda^{\ast}(\alpha^3(y))(\phi^{-4})^{\ast}(\xi)-\rho^{\ast}(\alpha^2(x))\lambda^{\ast}(\alpha^3(y))(\phi^{-4})^{\ast}(\xi)-\lambda^{\ast}(\alpha^2(y))\rho^{\ast}(\alpha^3(x))(\phi^{-4})^{\ast}(\xi),u
 \rangle\nonumber\\
 &&=\langle-\lambda^{\ast}(\alpha(x))\lambda^{\ast}(y)(\xi)-\rho^{\ast}(\alpha(x))\lambda^{\ast}(y)(\xi)-\lambda^{\ast}(\alpha^2(y))\rho^{\ast}(x)(\xi),u\rangle\nonumber
\end{eqnarray}
Hence, $\lambda^{\ast}(\alpha(y))\lambda^{\ast}(x)-\lambda^{\ast}(x\cdot y)(\phi^{-1})^{\ast}=\lambda^{\ast}(\alpha(y))(\rho^{\ast}+\lambda^{\ast})(x)+(\rho^{\ast}+\lambda^{\ast})(\alpha(x))\lambda^{\ast}(y).$ This ends the proof.
\end{proof}
\begin{corollary}
 Let $(A, \cdot ,\alpha)$ be a regular left Hom-pre-Jacobi-Jordan algebra. Then $(A,L^{\ast}+R^{\ast},-R^{\ast},(\alpha^{-1})^{\ast})$  is a representation of $(A, \cdot ,\alpha)$ where $L^{\ast}, R^{\ast}: A\rightarrow gl(A^{\ast})$ are defined by 
 \begin{eqnarray}
 && L_x^{\ast}\xi:=L_{\alpha(x)}^{\ast}(\alpha^{-2})^{\ast}(\xi) \mbox{  for all $x\in A, \xi\in A^{\ast},$ }\nonumber\\
  &&R_x^{\ast}\xi:=R_{\alpha(x)}^{\ast}(\alpha^{-2})^{\ast}(\xi) \mbox{  for all $x\in A, \xi\in A^{\ast}$ }\nonumber
 \end{eqnarray}
This representation  is called the
coadjoint representation.
\end{corollary}
also, we get:
\begin{corollary}
 Let $(A, \cdot ,\alpha)$ be a regular left Hom-pre-Jacobi-Jordan algebra. Then there is a natural left Hom-pre-Jacobi-Jordan algebra 
 $(A\oplus A^{\ast},\diamond, \alpha\oplus(\alpha^{-1})^{\ast}),$ where the product $\diamond$  is given by 
 \begin{eqnarray}
  &&(x_1+\xi_1)\diamond (x_2+\xi_2):=
  x_1\ast x_2+L_{x_1}^{\ast}\xi_2+ 
  R_{x_1}^{\ast}\xi_2 
  -R_{x_2}^{\ast}\xi_1
  %&&= x_1\ast x_2+ad_{\alpha(x_1)}^{\ast}(\alpha^{-2})^{\ast}(\xi_2)+
  %ad_{\alpha(x_2)}^{\ast}(\alpha^{-2})^{\ast}(\xi_1)
  \mbox{  for all $x_1,x_2\in A, \xi_1,\xi_2\in A^{\ast}.$ }\nonumber
 \end{eqnarray}
\end{corollary}
Similarly as in \cite{slls}, the following result can be proved.
\begin{proposition}\label{prop1}
Let $(V,\rho,\lambda, \phi)$ be a representation of a regular left Hom-pre-Jacobi-Jordan algebra 
 $(A, \cdot ,\alpha)$ such that $\phi$ is invertible. Then the dual representation of $(V, \rho^{\ast}, \lambda^{\ast},(\phi^{-1})^{\ast})$ is $(V,\rho,\lambda, \phi).$
\end{proposition}
Fom Theorem \ref{drplHJJ} and Proposition \ref{prop1}, it follows easily
\begin{proposition}
 Let $(V,\rho,\lambda, \phi)$ be a representation of a regular left Hom-pre-Jacobi-Jordan algebra 
 $(A, \cdot ,\alpha)$ such that $\phi$ is invertible. Then
the following conditions are equivalent:
 \begin{enumerate}
  \item[(i)] $(V,\rho+\lambda,-\lambda, \phi)$ is a representation of the
 left Hom-pre-Jacobi-Jordan algebra $(A, \cdot ,\alpha),$
 \item[(ii)] $(V, \rho^{\ast}, \lambda^{\ast},(\phi^{-1})^{\ast})$ is a representation of of the
 left Hom-pre-Jacobi-Jordan algebra
 $(A, \cdot ,\alpha),$
 \item[(iii)] $\mu(\alpha(x))\mu(y)=\mu(\alpha(y))\mu(x)$ for all $x,y\in A.$
 \end{enumerate}
\end{proposition}
\begin{definition}
A matched pair of left Hom-pre-Jacobi-Jordan algebras denoted by $(A_1,A_2,\rho_1,\lambda_1,\rho_2,\lambda_2),$ consists of two left Hom-pre-Jacobi-Jordan algebras $\mathcal{A}_1:=(A_1,\cdot,\alpha_1)$ and 
 $\mathcal{A}_2:=(A_2,\top,\alpha_2)$ together with representations $\rho_1, \mu_1: A_1\rightarrow gl(A_2)$ and 
 $\rho_2, \mu_2: A_2\rightarrow gl(A_1)$ with respect to $\alpha_2$ and $\alpha_1$ respectively such that for all $x,y,z\in A_1,$ $a,b,c\in A_2.$  the following conditions hold
 \begin{eqnarray}
  \rho_1(\alpha_1(x))(a\intercal b)&=&-\rho_1(\rho_2(a)x+\lambda_2(a)x)\alpha_2(b)-(\rho_1(x)a+\lambda_1(x)a)\intercal\alpha_2(b)\nonumber\\
 &&-\lambda_1(\lambda_2(b)x)\alpha_2(a)-\alpha_2(a)\intercal(\rho_1(x)b)\label{mppHJJ1}\\
 \lambda_1(\alpha_1(x))(a\circledast b)&=&-\alpha_2(a)\intercal(\lambda_1(x)b)-\alpha_2(b)\intercal(\lambda_1(x)a)
 -\lambda_1(\rho_2(a)x)\alpha_2(b)\nonumber\\
 &&-\lambda_1(\rho_2(b)x)\alpha_2(a)\label{mppHJJ2}\\
  \rho_2(\alpha_2(a))(x\cdot y)&=&-\rho_2(\rho_1(x)a+\lambda_1(x)a)\alpha_1(y)-(\rho_2(a)x+\lambda_2(a)x)\cdot\alpha_1(y)\nonumber\\
 &&-\lambda_2(\lambda_1(y)a)\alpha_1(x)-\alpha_1(x)\cdot(\rho_2(a)y)\label{mppHJJ3}\\
 \lambda_2(\alpha_2(a))(x\star y)&=&-\alpha_1(x)\cdot(\lambda_2(a)y)-\alpha_1(y)\cdot(\lambda_2(a)x)
 -\lambda_2(\rho_1(x)a)\alpha_1(y)\nonumber\\
 &&-\lambda_2(\rho_1(y)a)\alpha_1(x)\label{mppHJJ4}
 \end{eqnarray}
 where $\star$ is a product of the sub-adjacent Hom-Jacobi-Jordan algebra $A_1^C$ and $\circledast$ is a product of the sub-adjacent Hom-Jacobi-Jordan algebra $A_2^C.$
\end{definition}
\begin{theorem}
Let $\mathcal{A}_1:=(A_1,\cdot,\alpha_1)$ and 
 $\mathcal{A}_2:=(A_2,\top,\alpha_2)$ be left Hom-pre-Jacobi-Jordan algebras. Then,
 $(A_1,A_2,\rho_1,\lambda_1,\rho_2,\lambda_2),$ is a matched pair of left Hom-pre-Jacobi-Jordan algebras if and only if  $(A_1\oplus A_2, \diamond, \alpha_1\oplus\alpha_2)$ is a left Hom-pre-Jacobi-Jordan algebra where
\begin{eqnarray}
 (x+a)\diamond(y+b)&:=&(x\cdot y+\rho_2(a)y+\lambda_2(b)x)+(a\top b+\rho_1(x)b+\lambda_1(y)a) \label{oppHJJ1}\\
 (\alpha_1\oplus\alpha_2)(x+a)&:=&\alpha_1(x)+\alpha_2(a)\nonumber
\end{eqnarray}
\end{theorem}
\begin{proof} Fist, the commutativity of $\diamond$ and its multiplicativity with respect to $\alpha_1\oplus\alpha_2$ is equivalent to the multiplicativity of $\cdot$ and $\top$ with respect to $\alpha_1$ and $\alpha_2$ respectively and Condition (\ref{rHJJ1}). Next, let $x,y,z\in A_1$ and $a,b,\in A_2$, then by straightforward computations
\begin{eqnarray}
&& as_{\alpha_1\oplus\alpha_2}(x+a,y+b,z+c)\nonumber\\
&&=((x+a)\diamond(y+b))\diamond(\alpha_1\oplus\alpha_2)(z+c)+
 (\alpha_1\oplus\alpha_2)(x+a)\diamond((y+b)\diamond(z+c))\nonumber\\
 &&=(x\cdot y)\cdot\alpha_1(z)+(\rho_2(a)y)\cdot\alpha_1(z)+(\lambda_2(b)x)\cdot\alpha_1(z)+\rho_2(a\intercal b)\alpha_1(z)+\rho_2(\rho_1(x)b)\alpha_1(z)\nonumber\\
 &&+\rho_2(\lambda_1(y)a)\alpha_1(z)+\lambda_2(\alpha_2(c))(x\cdot y)+\lambda_2(\alpha_2(c))\rho_2(a)y+\lambda_2(\alpha_2(c))\lambda_2(b)x+(a\intercal b)\intercal\alpha_2(c)\nonumber\\
 &&+(\rho_1(x)b)\intercal\alpha_2(c)+(\lambda_1(y)a)\intercal\alpha_2(c)+\rho_1(x\cdot y)\alpha_2(c)+\rho_1(\rho_2(a)y)\alpha_2(c)
 +\rho_1(\lambda_2(b)x)\alpha_2(c)\nonumber\\
 &&+\lambda_1(\alpha_1(z))(a\intercal b)+\lambda_1(\alpha_1(z))\rho_1(x)b+\lambda_1(\alpha_1(z))\lambda_1(y)a
 +\alpha_1(x)\cdot(y\cdot z)+\alpha_1(x)\cdot(\rho_2(b)z)\nonumber\\
 &&+\alpha_1(x)\cdot(\lambda_2(c)y)+\rho_2(\alpha_2(a))(y\cdot z)+\rho_2(\alpha_2(a))\rho_2(b)z+\rho_2(\alpha_2(a))\lambda_2(c)y+\lambda_2(b\intercal c)\alpha_1(x)\nonumber\\
 &&+\lambda_2(\rho_1(y)c)\alpha_1(x)+\lambda_2(\lambda_1(z)b)\alpha_1(x)
 +\alpha_2(a)\intercal(b\intercal c)+\alpha_2(a)\intercal(\rho_1(y)c)
+\alpha_2(a)\intercal(\lambda_1(z)b)\nonumber\\
&&+\rho_1(\alpha_1(x))(b\intercal c)+\rho_1(\alpha_1(x))\rho_1(y)c+\rho_1(\alpha_1(x))\lambda_1(z)b+\lambda_1(y\cdot z)\alpha_2(a)+\lambda_1(\rho_2(b)z)\alpha_2(a)\nonumber\\
&&+\lambda_1(\lambda_2(c)y)\alpha_2(a).\nonumber
\end{eqnarray}
Switching $x+a$ and $y+b$ in the above expression of $as_{\alpha_1\oplus\alpha_2}(x+a,y+b,z+c)$, we obtain  $as_{\alpha_1\oplus\alpha_2}(y+b,x+a,z+c).$
Next, after rearranging terms, we obtain
\begin{eqnarray}
&&as_{\alpha_1\oplus\alpha_2}(x+a,y+b,z+c)+as_{\alpha_1\oplus\alpha_2}(y+b,x+a,z+c)
 =\Big(as_{\alpha_1}^t(x,y,z)+as_{\alpha_1}^t(y,x,z)\Big)\nonumber\\
 &&+\Big(as_{\alpha_2}^t(a,b,c)+as_{\alpha_2}^t(b,a,c)\Big)+
 \Big(\lambda_1(\alpha_1(z))\lambda_1(y)a+\lambda_1(y\cdot z)\alpha_2(a)+\lambda_1(\alpha_1(z))\rho_1(y)a
 \nonumber\\
 &&+\rho_1(\alpha_1(y))\lambda_1(z)a \Big)+
 \Big(\lambda_1(\alpha_1(z))\lambda_1(x)b+\lambda_1(x\cdot z)\alpha_2(b)+\lambda_1(\alpha_1(z))\rho_1(x)b+\rho_1(\alpha_1(x))\lambda_1(z)b \Big)\nonumber\\
 &&+\Big(\lambda_2(\alpha_2(c))\lambda_2(b)x+\lambda_2(b\intercal c)\alpha_1(x)+\lambda_2(\alpha_2(c))\rho_2(b)x+\rho_2(\alpha_2(b))\lambda_2(c)x \Big)+
 +\Big(\lambda_2(\alpha_2(c))\lambda_2(a)y\nonumber\\
 &&+\lambda_2(a\intercal c)\alpha_1(y)+\lambda_2(\alpha_2(c))\rho_2(a)y+\rho_2(\alpha_2(a))\lambda_2(c)y \Big)+
 \Big(\rho_1(x\star y)\alpha_2(c)+\rho_1(\alpha_1(x))\rho_1(y)c\nonumber
 \end{eqnarray}
 \begin{eqnarray}
 &&+ \rho_1(\alpha_1(y))\rho_1(x)c\Big)+
 \Big(\rho_2(a\circledast b)\alpha_1(z)+\rho_2(\alpha_2(a))\rho_2(b)z+ \rho_2(\alpha_2(b))\rho_2(a)z\Big)\nonumber\\
 &&+
 \Big( \rho_1(\alpha_1(y))(a\intercal c)+\rho_1(\rho_2(a)y+\lambda_2(a)y)\alpha_2(c)+(\rho_1(y)a+\lambda_1(y)a)\intercal\alpha_2(c)+\lambda_1(\lambda_2(c)y)\alpha_2(a)\nonumber\\
 &&+\alpha_2(a)\intercal(\rho_1(y)c)\Big)
 +\Big( \rho_1(\alpha_1(x))(b\intercal c)+\rho_1(\rho_2(b)x+\lambda_2(b)x)\alpha_2(c)+(\rho_1(x)b+\lambda_1(x)b)\intercal\alpha_2(c)\nonumber\\
 &&+\lambda_1(\lambda_2(c)x)\alpha_2(b)
 +\alpha_2(b)\intercal(\rho_1(x)c)\Big)
 +\Big(\lambda_1(\alpha_1(z)(a\circledast b)+\alpha_2(a)\intercal(\lambda_1(z)b)+\alpha_2(b)\intercal(\lambda_1(z)a)\nonumber\\
 &&+\lambda_1(\rho_2(a)z)\alpha_2(b)+\lambda_1(\rho_2(b)z)\alpha_2(a)   \Big)+
 \Big(\rho_2(\alpha_2(a))(y\cdot z)+\rho_2(\rho_1(y)a+\lambda_1(y)a)\alpha_1(z)\nonumber\\
 &&+(\rho_2(a)y+\lambda_2(a)y)\cdot\alpha_1(z)+\lambda_2(\lambda_1(z)a)\alpha_1(y)+\alpha_1(y)\cdot(\rho_2(a)z) \Big)+\Big(\rho_2(\alpha_2(b))(x\cdot z)\nonumber\\
 &&+\rho_2(\rho_1(x)b+\lambda_1(x)b)\alpha_1(z)+(\rho_2(b)x+\lambda_2(b)x)\cdot\alpha_1(z)+\lambda_2(\lambda_1(z)b)\alpha_1(x)+\alpha_1(x)\cdot(\rho_2(b)z) \Big)\nonumber\\
 &&+\Big(\lambda_2(\alpha_2(c))(x\star y)+\alpha_1(x)\cdot(\lambda_2(c)y)+\alpha_1(y)\cdot(\lambda_2(c)x)
 +\lambda_2(\rho_1(x)c)\alpha_1(y)+\lambda_2(\rho_1(y)c)\alpha_1(x) \Big).\nonumber
\end{eqnarray}
Thanks to  (\ref{rHJJ2}), (\ref{HpJJi}) and (\ref{rppHJJ2}), we obtain
\begin{eqnarray}
&&as_{\alpha_1\oplus\alpha_2}(x+a,y+b,z+c)+as_{\alpha_1\oplus\alpha_2}(y+b,x+a,z+c)
\nonumber\\ 
 &&=\Big( \rho_1(\alpha_1(y))(a\intercal c)+\rho_1(\rho_2(a)y+\lambda_2(a)y)\alpha_2(c)+(\rho_1(y)a+\lambda_1(y)a)\intercal\alpha_2(c)+\lambda_1(\lambda_2(c)y)\alpha_2(a)\nonumber\\
 &&+\alpha_2(a)\intercal(\rho_1(y)c)\Big)
 +\Big( \rho_1(\alpha_1(x))(b\intercal c)+\rho_1(\rho_2(b)x+\lambda_2(b)x)\alpha_2(c)+(\rho_1(x)b+\lambda_1(x)b)\intercal\alpha_2(c)\nonumber\\
 &&+\lambda_1(\lambda_2(c)x)\alpha_2(b)
 +\alpha_2(b)\intercal(\rho_1(x)c)\Big)
 +\Big(\lambda_1(\alpha_1(z)(a\circledast b)+\alpha_2(a)\intercal(\lambda_1(z)b)+\alpha_2(b)\intercal(\lambda_1(z)a)\nonumber\\
 &&+\lambda_1(\rho_2(a)z)\alpha_2(b)+\lambda_1(\rho_2(b)z)\alpha_2(a)   \Big)+
 \Big(\rho_2(\alpha_2(a))(y\cdot z)+\rho_2(\rho_1(y)a+\lambda_1(y)a)\alpha_1(z)\nonumber\\
 &&+(\rho_2(a)y+\lambda_2(a)y)\cdot\alpha_1(z)+\lambda_2(\lambda_1(z)a)\alpha_1(y)+\alpha_1(y)\cdot(\rho_2(a)z) \Big)+\Big(\rho_2(\alpha_2(b))(x\cdot z)\nonumber\\
 &&+\rho_2(\rho_1(x)b+\lambda_1(x)b)\alpha_1(z)+(\rho_2(b)x+\lambda_2(b)x)\cdot\alpha_1(z)+\lambda_2(\lambda_1(z)b)\alpha_1(x)+\alpha_1(x)\cdot(\rho_2(b)z) \Big)\nonumber\\
 &&+\Big(\lambda_2(\alpha_2(c))(x\star y)+\alpha_1(x)\cdot(\lambda_2(c)y)+\alpha_1(y)\cdot(\lambda_2(c)x)
 +\lambda_2(\rho_1(x)c)\alpha_1(y)+\lambda_2(\rho_1(y)c)\alpha_1(x) \Big).\nonumber
\end{eqnarray}
We deduce that (\ref{HpJJi}) holds in $A_1\oplus A_2$  if and only (\ref{mppHJJ1}), (\ref{mppHJJ2}), (\ref{mppHJJ3}) and (\ref{mppHJJ4}) holds.
\end{proof}
\begin{proposition}
Let $(A_1,A_2,\rho_1,\lambda_1,\rho_2,\lambda_2)$ be a matched pair of left Hom-pre-Jacobi-Jordan algebras
$\mathcal{A}_1:=(A_1,\cdot,\alpha_1)$ and $\mathcal{A}_2:=(A_2,\top,\alpha_2).$  Then,
$(A_1, A_2, \rho_1+\lambda_1,\rho_2+\lambda_2)$ is a matched pair of sub-adjacent Hom-Jacobi-Jordan algebras $A_1^C:=(A_1,\star,\alpha_1)$ and $A_2^C:=(A_2,\circledast,\alpha_2)$.
\end{proposition}
\begin{proof}
 Since $(A_1,\rho_2,\lambda_2,\alpha_1)$ and $(A_2,\rho_1,\lambda_1,\alpha_2)$ are representations of left Hom-pre-Jacobi-Jordan algebras $(A_2,\intercal,\alpha_2)$
 and $(A_1,\cdot,\alpha_1)$ respectively, it follows by Proposition \ref{sumrl}  that $(A_1,\mu_2:=\rho_2+\lambda_2,\alpha_1)$ and $(A_2,\mu_1:=\rho_1+\lambda_1,\alpha_2)$ are representations of sub-adjacent Hom-Jacobi-Jordan algebras $A_2^C$ and $A_1^C$ respectively.
  Next, let $x,y\in A_1$ and $a,b\in A_2.$ Then, by straightforward computations, after rearranging terms we get:
  \begin{eqnarray}
 && \mu_1(\alpha_1(x))(a\circledast b)+(\mu_1(x)a)\circledast\alpha_2(b)+(\mu_1(x)b)\circledast\alpha_2(a)
 +\mu_1(\rho_2(a)x)\alpha_2(b)+\mu_1(\rho_2(b)x)\alpha_2(a)\nonumber\\
 &&=
\Big(\rho_1(\alpha_1(x))(a\intercal b)+\rho_1(\rho_2(a)x+\lambda_2(a)x)\alpha_2(b)
+(\rho_1(x)a+\lambda_1(x)a)\intercal\alpha_2(b)
 +\lambda_1(\lambda_2(b)x)\alpha_2(a)\nonumber\\
 &&+\alpha_2(a)\intercal(\rho_1(x)b)\Big)
 + \Big(\rho_1(\alpha_1(x))(b\intercal a)+\rho_1(\rho_2(b)x+\lambda_2(b)x)\alpha_2(a)+(\rho_1(x)b+\lambda_1(x)b)\intercal\alpha_2(a)\nonumber\\
 &&
 +\lambda_1(\lambda_2(a)x)\alpha_2(b)+\alpha_2(b)\intercal(\rho_1(x)a)\Big)
  +\Big(\lambda_1(\alpha_1(x))(a\circledast b)+\alpha_2(a)\intercal(\lambda_1(x)b)+\alpha_2(b)\intercal(\lambda_1(x)a)
 \nonumber\\
 &&+\lambda_1(\rho_2(a)x)\alpha_2(b)+\lambda_1(\rho_2(b)x)\alpha_2(a)\Big)=0 
 \mbox{ ( by (\ref{mppHJJ1}), (\ref{mppHJJ2}) )}.\nonumber
 \end{eqnarray}
 Hence, we obtain (\ref{mpHJJ1}). Similarly, we compute
  \begin{eqnarray}
 && \mu_2(\alpha_2(a))(x\star y)+(\mu_2(a)x)\star\alpha_1(y)+(\mu_2(a)y)\star\alpha_1(x)
 +\mu_2(\rho_1(x)a)\alpha_1(y)+\mu_2(\rho_1(y)a)\alpha_1(x)\nonumber\\
 &&=
\Big(\rho_2(\alpha_2(a))(x\cdot y)+\rho_2(\rho_1(x)a+\lambda_1(x)a)\alpha_1(y)
+(\rho_2(a)x+\lambda_2(a)x)\cdot\alpha_1(y)
 +\lambda_2(\lambda_1(y)a)\alpha_1(x)\nonumber\\
 &&+\alpha_1(x)\cdot(\rho_2(a)y)\Big)
 + \Big(\rho_2(\alpha_2(a))(y\cdot x)+\rho_2(\rho_1(y)a+\lambda_1(y)a)\alpha_1(x)+(\rho_2(a)y+\lambda_2(a)y)\cdot\alpha_1(x)\nonumber\\
 &&
 +\lambda_2(\lambda_1(x)a)\alpha_1(y)+\alpha_1(y)\cdot(\rho_2(a)x)\Big)
  +\Big(\lambda_2(\alpha_2(a))(x\star y)+\alpha_1(x)\cdot(\lambda_2(a)y)+\alpha_1(y)\cdot(\lambda_2(a)x)
 \nonumber\\
 &&+\lambda_2(\rho_1(x)a)\alpha_1(y)+\lambda_2(\rho_1(y)a)\alpha_1(x)\Big)=0 
 \mbox{ ( by (\ref{mppHJJ3}), (\ref{mppHJJ4}) )}.\nonumber
 \end{eqnarray}
Then, we get also (\ref{mpHJJ2}).
\end{proof}
For the next result, we suppose that there is a left Hom-pre-Jacobi-Jordan algebra structure on $A^{\ast}$ and we denote $\mathcal{L}$ and $\mathcal{R}$ the corresponding operations. Now, we can prove
\begin{proposition}
 Let $(A,\cdot,\alpha)$ and $(A^{\ast},\intercal,(\alpha^{-1})^{\ast})$ be two left Hom-pre-Jacobi-Jordan algebras. Then,
 $(A^,(A^{\ast})^C, L^{\ast}, \mathcal{L}^{\ast})$ is a matched pair of Hom-Jacobi Jordan algebras if and only if
 $(A,A^{\ast},L^{\ast}+R^{\ast},-R^{\ast},\mathcal{L}^{\ast}+\mathcal{R}^{\ast},-\mathcal{R}^{\ast} )$ is a matched pair of
 left pre-Hom-Jacobi-Jordan algebras.
\end{proposition}
\begin{definition}
 Let $(V, \rho,\lambda,\phi)$ is a representation of a left Hom-pre-Jacobi-Jordan algebra $(A,\cdot, \alpha).$ A linear
map $T : V\rightarrow A$ is called an $\mathcal{O}$-operator of $(A,\cdot, \alpha)$ associated to $(V, \rho,\lambda,\phi)$ if
\begin{eqnarray}
 && T\phi=\alpha T\label{rbHpJJ1}\\
 &&T(u)\cdot T(v)=T\Big(\rho(T(u))v+\lambda(T(v)u\Big)\mbox{ for all $u,v\in V$} \label{rbHpJJ2}
\end{eqnarray}
\end{definition}
\begin{example}
 Let $(A, \cdot, \alpha)$ be a left Hom-pre-Jacobi-Jordan algebra and $(V,\rho, \phi)$ be a representation of $(A, \cdot, \alpha).$
It is easy to verify that $A\oplus V$ is a representation of $(A, \cdot, \alpha)$ under the maps $\rho_{A\oplus V}: A\rightarrow gl(A\oplus V)$ defined by
\begin{eqnarray}
 &&\rho_{A\oplus V}(a)(b+v):=a\cdot b+\rho(a)v.\nonumber
\end{eqnarray}
Define the linear map $T: A\oplus V\rightarrow A, a+v\mapsto a.$ Then $T$ is an $\mathcal{O}$-operator of $A$ associated to the representation $(A\oplus V,\rho_{A\oplus V},\alpha\oplus\phi).$
\end{example}
Let give another example of $\mathcal{O}$-operators of left Hom-pre-Jacobi-Jordan algebras.
%\begin{example}  
%\end{example}
It is easy to prove:
\begin{proposition}
 If $T$ is an $\mathcal{O}$-operator of a left Hom-pre-Jacobi-Jordan $(A,\cdot, \alpha)$ associated to a representation $(V, \rho,\lambda,\phi)$
then $T$ is an $\mathcal{O}$-operator of its associated Hom-Jacobi-Jordan $(A,\star, \alpha)$ associated to the representation $(V, \rho+\lambda,\phi).$
\end{proposition}
As Hom-associative algebras case \cite{tcsmam}, let give some characterizations of $\mathcal{O}$-operators on left Hom-pre-Jacobi-Jordan algebras.
\begin{proposition}
 A linear map $T : V\rightarrow A$ is an $\mathcal{O}$-operator associated to a representation $(V,\rho,\phi)$ of a left Hom-pre-Jacobi-Jordan  algebra $(A, \cdot, \alpha)$ if and only if the graph of $T,$
$$G_r(T):=\{(T(v), v), v\in  V\}$$
is a subalgebra of the semi-direct product algebra $A\ltimes V.$
\end{proposition}
The following result shows that an $\mathcal{O}$-operator can be lifted up the
Rota-Baxter operator.
\begin{proposition}
 Let $(A, \cdot, \alpha)$ be a left Hom-pre-Jacobi-Jordan algebra, $(V,\rho,\phi)$ be a representation of $A$ and
$T : V\rightarrow A$ be a linear map. Define 
$\widehat{T}\in End(A\oplus V)$ by 
$\widehat{T}(a+v):=T(v).$ Then T is an $\mathcal{O}$-operator associated to $(V,\rho,\phi)$
if and only if $\widehat{T}$ is a Rota-Baxter operator on $A\oplus V.$
\end{proposition}
In the sequel, let give  some results about $\mathcal{O}$-operators.
\begin{proposition}\label{lHLPRB}
 Let $T$ be an $\mathcal{O}$-operator on a left Hom-pre-Jacobi-Jordan algebra $(A, \cdot, \alpha)$ with respect to a representation  $(V,\rho,\phi ).$ If define a map $\diamond$ by 
 \begin{eqnarray}
 &&u\intercal v:=\rho(T(u))v+\lambda(T(v))u \mbox{ for all $(u, v)\in V^{\times 2}$}\nonumber
 \end{eqnarray}
then, $(V,\intercal, \phi)$ is a left Hom-pre-Jacobi-Jordan algebra. Moreover, 
$T$ is a morphism from the left Hom-pre-Jacobi-Jordan algebra 
 $(V,\intercal,\phi)$ to the initial left Hom-pre-Jacobi-Jordan algebra $(A, \cdot, \alpha).$
\end{proposition}
\begin{proof}
For all $u,v,w\in V,$ rearranging terms after straightforward computations, we obtain
\begin{eqnarray}
 && as_{\phi}^t(u,v,w)+as_{\phi}^t(v,u,w)=
 (u\intercal v)\intercal\phi(w)+\phi(u)\intercal(v\intercal w)+(v\intercal u)\intercal\phi(w)+\phi(v)\intercal(u\intercal w)\nonumber\\
 &&=
 \Big(\rho(T(u)\cdot T(v))\phi(w)+\rho(T(v)\cdot T(u))\phi(w)+\rho(T(\phi(u)))\rho(T(v))w+\rho(T(\phi(v)))\rho(T(u))w\Big)\nonumber\\
 &&+\Big(\lambda(T(\phi(w)))\lambda(T(v))u+ \lambda(T(v)\cdot T(w))\phi(u)+\lambda(T(\phi(w)))\rho(T(v))u+\rho(T(\phi(v)))\lambda(T(w))u\Big)\nonumber\\
 &&+\Big(\lambda(T(\phi(w)))\lambda(T(u))v+ \lambda(T(u)\cdot T(w))\phi(v)+\lambda(T(\phi(w)))\rho(T(u))v+\rho(T(\phi(u)))\lambda(T(w))v\Big)\nonumber\\
 &&=0 \mbox{ ( by  (\ref{rbHpJJ1}), (\ref{rHJJ2}) and (\ref{rppHJJ2}).}\nonumber
\end{eqnarray}
Hence, $(V,\intercal, \phi)$ is a left Hom-pre-Jacobi-Jordan algebra. The second assertion follows from (\ref{rbHpJJ2}) and the definition of $\intercal.$
\end{proof}
In order to give another characterization of $\mathcal{O}$-operators, let introduce the following:
\begin{definition}\label{njpHJJ}
 Let $(A, \cdot, \alpha)$ be a left Hom-pre-Jacobi-Jordan algebra. A linear map 
 $N : A\rightarrow A$ is
said to be a Nijenhuis operator if $N\alpha=\alpha N$ and its Nijenhuis torsions vanish, i.e.,
\begin{eqnarray}
 && N(x)\cdot N(y)=N(N(x)\cdot y + x\cdot N(y)-N(x\cdot y)), \mbox{ for all $x, y\in A,$}\nonumber
\end{eqnarray}
Observe that the deformed multiplications
$\ast_N: A\oplus A\rightarrow A$ given by
\begin{eqnarray}
 && x\cdot_N y:= N(x)\cdot y+x\cdot N(y)-N(x\cdot y),\nonumber
\end{eqnarray}
gives rise to a new left Hom-pre-Jacobi-Jordan multiplication on $A,$ and $N$ becomes a 
morphism from the left Hom-pre-Jacobi-Jordan algebra $(A,\cdot_N,\alpha)$ to the initial left Hom-pre-Jacobi-Jordan algebra $(A, \cdot, \alpha).$
\end{definition}
Now, we can esealy check the following result.
\begin{proposition}
 Let $\mathcal{A}:=(A, \cdot, \alpha)$ be a left Hom-pre-Jacobi-Jordan algebra and $\mathcal{V}:=(V,\rho, \phi)$ be a representation of $(A, \cdot, \alpha).$
 A linear map $T: V\rightarrow A$ is an
  $\mathcal{O}$-operator on $\mathcal{A}$ with respect to the
$\mathcal{V}$ if and only if $N_T:=\left(
\begin{array}{cc}
 0& T\\
 0& 0
\end{array}
\right)
: A\oplus V \rightarrow A\oplus V$ is a Nijenhuis operator on
the left Hom-pre-Jacobi-Jordan algebra $A\oplus V.$
\end{proposition}
 \begin{proposition}
  If $N$ is a Nijenhuis operator on a left Hom-pre-Jacobi-Jordan algebra $(A,\cdot, \alpha),$ then $N$ is a
Nijenhuis operator on the sub-adjacent Hom-Jacobi-Jordan algebra $A^C.$
 \end{proposition}
\begin{proof}
 Since $N$ is a Nijenhuis operator on $(A,\cdot,\alpha),$ we have $N\alpha=\alpha N.$ For all $x, y \in A,$ by
Definition \ref{njpHJJ}, we obtain
\begin{eqnarray}
 &&N(x)\star N(y)= N(x)\cdot N(y)+N(y)\cdot N(x)\nonumber\\
 &&=N(N(x)\cdot y + x\cdot N(y)-N(x\cdot y)+ (N(y)\cdot x + y\cdot N(x)-N(y\cdot x))\nonumber\\
 &&=N(N(x)\star y+y\star N(x)-N(x\star y)).\nonumber
\end{eqnarray}
\end{proof}

\vspace*{1cm}
Sylvain Attan\\
 D\'epartement de Math\'ematiques, Universit\'{e} d'Abomey-Calavi
01 BP 4521, Cotonou 01, B\'enin. E.mail: syltane2010@yahoo.fr
\end{document}